\documentclass[11pt]{amsart}
\usepackage[english]{babel}
  \usepackage{amsmath,amsfonts,amssymb,relsize,amsthm}
  \usepackage{esint}
  \usepackage{dsfont}

\usepackage{tikz}
\usepackage{pgfplots}
  \usepackage{hyperref} 

 \usepackage[active]{srcltx} 
 \usepackage{color,soul} 

\usepackage{geometry}
\newtheorem{theorem}{Theorem}[section]

\newtheorem{lemma}[theorem]{Lemma}
\newtheorem{proposition}[theorem]{Proposition}

\theoremstyle{definition}

\newtheorem{remark}{Remark}

\numberwithin{equation}{section}

\def\O{{\Omega}}

\def\eps{{\varepsilon}}
\def\a{{\mathcal{A}}}

\def\lg{{\mathfrak{L}}}

\def\m{{\mathcal{M}}}
\def\ms{{\mathcal{M}^*}}

\def\r{{\mathcal{R}}}

\def\E{{\mathcal{E}}}
\def\e{{\mathcal{E}}}
\def\L{{\mathcal{L}}}
\def\M{{\mathcal{M}}}

\def\df{{\mathfrak{D}_{x}}}
\def\ddf{{\mathfrak{D}_{xx}}}

\def\R{{\mathbb{R}}}

\def\N{{\mathbb{N}}}

\newcommand{\ope}[1]{\E[{#1}]}

\newcommand{\opr}[1]{\r[{#1}]}
\newcommand{\oplg}[1]{\lg[{#1}]}

\newcommand{\oplb}[2]{\L_{_{#2}}[{#1}]}

\newcommand{\opm}[1]{\M[{#1}]}
\newcommand{\opms}[1]{\M^*[{#1}]}
\newcommand{\mbs}[1]{\M^*_{_{#1}}}
\newcommand{\opmbs}[2]{\M^*_{_{#2}}[{#1}]}
\newcommand{\mb}[1]{\M_{_{#1}}}
\newcommand{\opmb}[2]{\M_{_{#2}}[{#1}]}

\newcommand{\opdf}[1]{\df[{#1}]}
\newcommand{\opddf}[1]{\ddf[{#1}]}

\newtheorem{claim}[theorem]{Claim}

\newenvironment{formula}[1]{\begin{equation}\label{eq:#1}}
                       {\end{equation}\noindent}

\def\Fi#1{\begin{formula}{#1}}
\def\Ff{\end{formula}\noindent}

\setstcolor{red}

\def\ds{\displaystyle}

\title{Can a population survive in a shifting environment using non-local  dispersion  
}

\author[J\'er\^{o}me Coville]{J\'er\^ome Coville}
\address{BioSP, INRAE, 84914, Avignon, France}
\email{jerome.coville@inrae.fr}

\begin{document}

\maketitle
\begin{abstract}
In this article,  we analyse  the non-local   model :
$$
 \partial_t U(t,x)=J\star U(t,x) -U(t,x) + f(x-ct,U(t,x)) \quad \textrm{for }t>0,\textrm{ and } x \in \R,
$$
where $J$ is a positive continuous dispersal kernel and $f(x,s)$ is a heterogeneous KPP type non-linearity describing the growth rate of the population.  The ecological niche of the population is assumed to be bounded  (i.e. outside a compact set, the environment is assumed to be lethal for the population) and shifted through time at a  constant speed $c$. For compactly supported dispersal kernels $J$, assuming that for $c=0$ the population survive,  we prove that there exists a critical speeds $c^{*,\pm}$ and $c^{**,\pm}$  such that for all $-c^{*,-}< c < c^{*,+}$ then the population will survive and will perish when $ c\ge c^{**,+}$ or $c\le -c^{**,-}$. To derive this results we first obtain an optimal persistence criteria depending of the speed $c$ for non local problem with a drift term. Namely, we prove that for a positive speed $c$ the population persists  if and only if the generalized principal eigenvalue $\lambda_p$  of the linear problem
$$ c\opdf{\varphi}+ J\star \varphi -\varphi + \partial_sf(x,0)\varphi+\lambda_p\varphi=0 \quad \text{ in }\quad \R,$$
 is negative. $\lambda_p$ is a spectral quantity that we defined in the spirit of the generalized first eigenvalue of an elliptic operator. 
The speeds $c^{*,\pm}$ and $c^{**,\pm}$ are then obtained through a fine analysis of the properties of $\lambda_p$ with respect to $c$. In particular, we  establish  its  continuity  with respect to the speed $c$.  In addition,  for any continuous bounded  non-negative initial data, we establish the long time behaviour of the solution  $U(t,x)$.            
\end{abstract}

\section{Introduction}
Environmental changes due to earth global warming are  enforcing species to shift their ranges to more favorable  habitat region e.g. to  the north or upward in elevation \cite{Hughes2000, McCarty2001, Parmesan1999, Walther2002}.  
The understanding of this transition and its consequence  on the global diversity is of prime interest  and  numerous type of models have been considered to understand, evaluate and/or highlight this complex dynamical process and its  impact \cite{Bakkenes2002,Berestycki2009,Berestycki2008,Box1981,Guisan2005,Harsch2014,Potapov2004,Zhou2011}.
   
  In this article, we are interested in the influence of long range dispersal processes for species living in a shifting environment caused by such environmental changes. For such a model species, we can think of trees of which seeds and pollens are disseminated on a wide range and whose  environment  changes through time as a consequence of a climate change. This possibility of a long range dispersal is well known in ecology, where numerous data now available support this assumptions \cite{Cain2000,Clark1998,Clark1998a,Schurr2008}. 

One way to model long range dispersal in the context of a shifting environment is to consider a population described by a shifted  integrodifference model, that is,  modelling the population by a density  $n(t,x)$  whose time evolution (growth and dispersal) is governed by the following discrete in time equation :
\begin{equation}\label{cdm-ide}
n_{t+1}(x)=\int_{-\frac{L}{2}+ct}^{\frac{L}{2}+ct}J(x,y)f[n_{t}(y)]\,dy,
\end{equation}  
 where $J$ is a dispersal kernel describing the  movement of individuals between the time $t$ and $t+1$, $c>0$ is the speed of the environmental changes and $f$ is a growth function describing the demography of the species. 

In this setting,  Zhou and its collaborators  show in \cite{Zhou2011,Zhou2013} how the speed of change  $c$ affects the persistence of the population. Showing in particular that for too large speed of change $c$, the population will then go extinct whereas for small value of $c$, the species will be able to adapt and survive. They also  suggest the existence of a critical speed $c^*>0$ for which for all speed $c< c^*$ then the population survive and goes extinct if $c\ge c^*$.  For particular kernels and growth function, they provide some heuristics to compute this $c^*$.

Another commonly used model that integrates such long range dispersal  is the following {\em nonlocal reaction diffusion equation} (\cite{Fife1996,Grinfeld2005,Hutson2003,Lutscher2005,Turchin1998}): 
 \begin{equation}
 \partial_t U(t,x)=J\star U(t,x) -U(t,x) + f(x-ct,U(t,x)) \quad \textrm{for }t>0,\textrm{ and } x \in \R. \label{cdm-eq-dyn}
 \end{equation}
Here also $U(t,x)$ is  the density of the considered population, $J$ is a dispersal kernel, $f(x,s)$ is a KPP type non-linearity describing 
the growth rate of the population and $c$ is the speed of the environmental changes. Compare to \eqref{cdm-ide}, the  equation \eqref{cdm-eq-dyn} is defined for all times $t$  and the movement and growth processes are not any more entangled. \\  
In this particular setting the tail of the kernel can be thought of as the range of dispersion or as a measure of  the frequency at which  long dispersal events occur,  reflecting  at the population level  some intrinsic variability in the capacity of each individual to disperse \cite{Hapca2009,Petrovskii2011}.

The aim of this article is to investigate, as for integrodifference models,  how the speed of the environmental change affects the survival of the population modelled by \eqref{cdm-eq-dyn}.   
Throughout this paper we will always make the following assumptions on the dispersal kernel $J$.
\begin{equation}\label{hypj1}
\tag{\textbf{H1}}
J \in C(\R)\cap L^1(\R) \text{is nonnegative and of unit mass (i.e.$\int_{\R}J(z)dz=1$)}.
\end{equation} 

\begin{equation}
\label{hypj2}
 \tag{\textbf{H2}} \text{$J(0)>0$} 
\end{equation}

There are many ways to model the impact of a changing environment. For example,  Li, Wang and  Zhao \cite{Li2018} consider an habitat which gradually change from a bad environment to a good environment. Namely, they consider $f(x-ct,s):=s(r(x-ct) -s)$ with $r(z)$ a continuous non decreasing function such that $\ds{\lim_{z\to -\infty}r(z)<0<\lim_{z\to +\infty} r(z)}$.  For such shrinking environment and for thin tailed kernel (i.e. $J$ such that $\int_{\R}J(z)e^{\lambda z}\,dz<+\infty$)the authors prove the existence of a critical speed $c^*$ defining the threshold between  persistence and extinction.  
That is for a speed of change $c\ge c^*$ then the population will not survive whereas a monotone front will exist when $c<c^*$. 
Moreover the critical speed is defined by the following spectral formula :
\begin{equation}\label{cdm-eq:c*}
c^*:=\inf_{\lambda>0}\frac{1}{\lambda}\left(\int_{\R}J(z)e^{\lambda z}-1+\sup_{x\in\R}r(x)\right).
\end{equation}       

Here, instead of a monotone habitat whose favourable part shrinks through time, we focus our analysis on species that have a bounded ecological niche that shift with speed $c$.  A simple way to model  such a spatial repartition consists in considering  that the environment is hostile to the species outside a bounded set. For instance, biological populations that are sensitive to temperature thrive only in a limited latitude zone. Thus, if $x$ is the latitude, we get such dependence.  This fact is translated in our model  by assuming  that $f$ satisfies:

\begin{equation}\label{hypf1}
\text{$f \in C(\R\times[0,+\infty))$  is smooth and of KPP type }
\tag{\textbf{H3}} 
\end{equation}
that is 
$$ 
\begin{cases}
&\hbox{$ f$ is differentiable with respect to $s$  },\\
& \hbox{$\forall \, s\ge 0,\; f(\cdot,s)\in C^{0,1}(\R)$},\\
& \hbox{$f(\cdot,0)\equiv0$},
\\
& \hbox{For all $x\in \R, \,f(x,s)/s$ is  decreasing  with
respect to $s$ on $(0,+\infty)$}.
\\
& \hbox{There exists $S(x)\in C(\R)\cap L^{\infty}(\R)$ such that $f(x,S(x))\leq 0$ 
for all $ x\in \R$.}
\end{cases}
$$
 
  \begin{equation}\label{hypf2}\tag{\textbf{H4}} 
  \hbox{ $\limsup_{|x|\to \infty} \frac{f(x,s)}{s}<0,$\quad \text{ uniformly in } \quad $s\ge 0.$}    
   \end{equation}
   
and  $f_s(x,\cdot)$ satisfies this uniform Lipschitz condition

\begin{equation}\label{hypf3}\tag{\textbf{H5}} 
\sup_{x\in \R}\left(\sup_{s_0,s_1\ge 0}\frac{|f_s(x,s_0)-f_s(x,s_1)|}{|s_0-s_1|}\right)<+\infty  
\end{equation}

A typical example of such a non linearity is given by $f(x,s):=s(a(x) - b(x)s)$ with $b(x)>0, b\in L^{\infty}$ and $a(x)$ satisfying
$\limsup_{|z|\to \infty} a(z)<0$.
This structure of the environment prohibits the existence of monotone fronts and the analysis provided in \cite{Li2018} no longer holds true.

This type of moving environments were recently investigated in \cite{DeLeenheer2020}  for a  very specific type of  nonlinearity $f(x,s)$ presenting some symmetry.
Namely, assuming that $f(x,s)$ satisfies the conditions:
$$ 
\begin{cases}
&\hbox{$f\in C^{1}(\R^2)$ and  there are $a,q,L, L_0>0$ and $\phi^+,\phi^-\in C^1(\R)$ such that   },\\
& \hbox{$\forall \,s,\, \forall |x|\ge L+L_0,\; f(x,s)=-q s$},\\
& \hbox{$\forall \,s,\, \forall |x|\le L,\;f(x,s)=s(a-s)$},
\\
& \hbox{$\forall \,s,\, \forall L\le x\le L+L_0,\;f(x,s)=-qs +s[a-s +q]\phi^{+}\left(\frac{x-L}{L_0}\right)$},
\\
& \hbox{$\forall \,s,\, \forall -L-L_0\le x\le -L,\;f(x,s)=-qs +s[a-s +q]\phi^{-}\left(\frac{x-L}{L_0}\right)$},
\end{cases}
$$
where $\phi^+$ and $\phi^-$ are respectively smooth regularisation of the characteristic function of $\R^{-}$ and of $\R^+$,  such that $supp( \phi^+) \subset (-\infty,1), \phi^+_{|\R^-}\equiv 1$ and $\phi^-(x)=\phi^+(-x)$.

For such specific nonlinearities and assuming that $J$ is thin tailed and symmetric,  the authors in \cite{DeLeenheer2020} prove the existence of critical speed $c^*>0$ determined by the maximal linearised growth rate (i.e. $c^*$ defined by \eqref{cdm-eq:c*} with $\sup_{x\in \R} r(x)=a$) such that for all $c\ge c^*$ the population go extinct whereas for $0<c<c^*$, the population persistence is dependant of the patch size habitat $L$, i.e., the population survive if the patch size is greater that some critical value $L^*$ otherwise the population dies. To obtain these results  the authors rely on  some notion of persistence for the solution $u(t,x)$ of the Cauchy problem \eqref{cdm-eq-dyn}, namely the population is said to persist if for any initial datum $u_0$ such that $\ds{\inf_{\R}u_0>0}$, then the corresponding solution of the Cauchy problem \eqref{cdm-eq-dyn} with initial data $u_0$ satisfies for all compact $K\subset \R$,  
$$\liminf_{t\to +\infty} \inf_{x-ct \in K}u(t,x)>0.$$ 

 Our main purpose here is to extend the above results to more general situations by finding more generic conditions on $J$, $c$ and $f$ that characterise the persistence of the species modelled by \eqref{cdm-eq-dyn}. 
   
 In this task, as in \cite{DeLeenheer2020,Li2018}  we focus our analysis on the description of  positive \eqref{cdm-eq-dyn} defined in a moving frame of speed $c$. That is we look for solution $\tilde U(t,x):=U(t,x+ct)$ that satisfies 
\begin{equation}
\partial_t\tilde U(t,\xi)-c\mathfrak{D}_{\xi}[\tilde U](t,\xi)=\int_{\R}J(\xi -\xi')\tilde U(t,\xi')\,dy - \tilde U(t,\xi)+f(\xi,\tilde U(t,\xi)) \quad \text{ for }\quad t>0,\; \xi\in \R, \label{cdm-eq-parab}
\end{equation}
 where we set  $\xi:=x+ct$ as a new variable and we use the Euler's notation to denote differential operators $\frac{d}{dx}f(x)$, i.e. $\opdf{f}=\frac{d}{dx}f(x)$.    
In particular, we look  for stationary solutions $u$ of \eqref{cdm-eq-parab}  which are then  positive solutions of the equation below 
\begin{equation}
 -c\mathfrak{D}_{x}[u](x)=J\star u(x) -u(x) + f(x,u(x)) \quad \text{ for  }\quad x \in  \R, \label{cdm-eq}
 \end{equation}
Existence of such stationary solutions is naturally expected to provide the right persistence criterion. We will see that this is indeed the case. 

In the literature,  such problem   have been well studied for the local reaction diffusion version of \eqref{cdm-eq-dyn} 
 \begin{equation}\label{cdm-eq-rd}
 \partial_t U(t,z)=\Delta U(t,z) + f(z_1-ct,y,U(t,z))\quad \textrm{for }t>0,\textrm{ and } z \in \O, 
 \end{equation}
 where $\O$ is cylindrical domain of $\R^N$, possibly $\R^N$ itself see for example  \cite{Berestycki2008,Berestycki2009,Berestycki2009b,Vo2015}.
For such reaction diffusion equations the persistence criteria are often derived from the sign of the first eigenvalue of the linearised problem  at the 0 solution. One is thus led to determine the sign of the first eigenvalue  $\lambda_1(\Delta+c\partial_{x_1}+ \partial_sf(x_1,y,0),\O)$ of the spectral problem 
\begin{equation}\label{cdm-eq-rd-lin0}
\Delta \varphi +c\partial_{x_1}\varphi +\partial_sf(x_1,y,0)\varphi +\lambda_1 \varphi=0 \quad \text{ in }\quad \O 
 \end{equation}
 associated with the proper boundary conditions (if $\Omega\not=\R^N$).
 
In the above situation the existence of a positive stationary solution to \eqref{cdm-eq-rd} is uniquely conditioned by the sign of $\lambda_1$. More precisely, there exists  a unique positive stationary solution if and only if  $\lambda_1<0$. If such type of criteria seems reasonable for problems defined on bounded set, it is less obvious for problems in unbounded domains.  In particular, in unbounded domains, one of the main difficulty concerns the definition of $\lambda_1$.  As shown in \cite{Berestycki2006,Berestycki2007,Berestycki2015},  the notion of first eigenvalue in unbounded domain can be quite delicate and several definitions of $\lambda_1$ exist rendering the question of  sharp persistence criteria already quite involved. When such existence criteria is established, it is then possible to investigate its behaviour with respect to the speed $c$ and for the above problem, using the Liouville transform, it is easy to make explicit the dependence of $\lambda_1$ with respect to $c$. Indeed, in such situation we have 

$$ \lambda_1\left(\Delta+c\partial_{x_1}+\partial_sf(x_1,y,0)\right)=\lambda_1\left(\Delta+\partial_sf(x_1,y,0)\right)+\frac{c^2}{4}.$$
 From the above formula, it is then easy to derive the critical speed for which a species can survive. 
  
Much less is known for the non-local equation \eqref{cdm-eq}  and, persistence criteria have been essentially investigated  
in some specific situations such as periodic media : \cite{Coville2008b,Coville2013,Shen2012,Rawal2012,Shen2015} or for a version of the problem  \eqref{cdm-eq}  with no time dependence \cite{Bates2007,Berestycki2016a,Coville2010,Coville2015,Garcia-Melian2009,Kao2010,Shen2012} : 
\begin{equation}\label{cdm-eq-bounded}
 \partial_t U(t,x)=\int_{\O}J(x-y)U(t,y)\,dy -U(t,x) + f(x,U(t,x)) \quad \textrm{for }t>0,\textrm{ and } x \in \O. 
 \end{equation}
Similarly to the local diffusion case, for $KPP$ like non-linearities, the existence of a positive  solution of the non-local equation \eqref{cdm-eq-bounded} can be characterised  by the sign of a spectral quantity $\lambda_p$, called  the generalised principal eigenvalue or the spectral point of 
\begin{equation}\label{cdm-eq-lin0}
\int_{\O}J(x-y) \varphi(y) \, dy - \varphi +\partial_sf(x,0)\varphi +\lambda \varphi=0 \quad \text{in}\quad\O. 
 \end{equation}
In the spirit of \cite{Berestycki1994}, this generalised principal eigenvalue $\lambda_p$ can be defined by :

$$\lambda_p:=\sup\left\{\lambda\in\R\, |\, \exists \varphi \in C(\O), \varphi>0,\; \text{ such that }\;\opmb{\varphi}{\O} +\partial_sf(x,0)\varphi +\lambda \varphi\le 0 \quad \text{in}\quad \O \right\},$$ 
where $\opmb{\varphi}{\O}$ denotes  
$$\opmb{\varphi}{\O}:= \int_{\O}J(x-y)\varphi(y)\,dy -\varphi.$$

\noindent This is only very recently that some progress have been made on the spectral theory of nonlocal operators with a drift term, that is operators of the type 
$\ds{ c\df + \mb{\O} + {\bf a}}$ see for example \cite{Altenberg2012,Coville2017a,Coville2020,DeLeenheer2020,Li2018, Velleret2019}. This new understanding of such spectral problems provide an adequate framework   to the analysis of for complex media such as those described in \cite{Li2018} or those we study in the present paper. 
As in \cite{Berestycki2016a,Coville2017a, Coville2020}, for the operator 
$\ds{ c\df + \mb{\O} + {\bf a}}$ let us define the  quantity
$$\lambda_p\left(c\df +\mb{\O} +{\bf a}\right):=\sup \left\{\lambda\in \R\,|\, \exists \varphi \in C^1(\R^N), \varphi>0\; \text{ such that }\; c\opdf{\varphi}+ \opmb{\varphi}{\O}+ a(x)\varphi + \lambda\varphi\le 0\right\}.$$
Equipped with this notion, we can now state our main results. In the first one we establish a simple sharp persistence criteria assuming that the dispersal kernel $J$ has a compact support.    

\begin{theorem}\label{cdm-thm1}
Assume that $J,f$ satisfy \eqref{hypj1}-\eqref{hypf2} and assume further that $J$ is compactly supported and $c>0$.  Then,  there exists a  positive solution, $u$, of \eqref{cdm-eq} if and only if $\lambda_p(c\df+\M+ {\bf \partial_sf(x,0)})<0$, where $\M$ denotes the continuous operator $\opm{\varphi}=J\star \varphi-\varphi$ and

$$\lambda_p(c\df+\M+{\bf \partial_sf(x,0)}):=\sup \{\lambda\in \R\,|\, \exists \varphi \in C^1(\R), \varphi>0\; \text{ such that }\; c\opdf{\varphi}+\opm{\varphi}+\partial_sf(x,0)\varphi+\lambda\varphi\le 0\}.$$
When it exists, the solution is unique, that is, 
if $v$ is another bounded solution, then $u=v$ almost everywhere. Moreover, for any non-negative initial data $u_0 \in C(\R)\cap L^{\infty}(\R)$ we have the following asymptotic behaviour:
\begin{itemize}
\item If $\lambda_p(c\df+\M+\partial_sf(x,0))\ge 0$, then  the solution satisfies $\|U(t,\cdot)\|_{\infty}\to 0$ as $t\to \infty$,
\item If $\lambda_p(c\df+\M+\partial_sf(x,0))< 0$, then  the solution satisfies $\|U(t,\cdot)- u(\cdot-ct)\|_{\infty}\to 0$ as $t\to \infty$.
\end{itemize}
\end{theorem}

Next,   we establish  some   properties of the quantity $\lambda_p(c\df+\M+{\bf \partial_sf(x,0)})$ that will help to analyse the dependence of $\lambda_p$ with respect to the speed $c$.  Namely,  

\begin{theorem}\label{cdm-thm3}
Assume that $J,f$ satisfy \eqref{hypj1}-\eqref{hypf2} and assume further that $J$ is compactly supported.  Then the map 
$$\begin{array}{llll}
   &\R^+ &\to     & \R \\
     & c &\mapsto & \lambda_p(c\df+\M+ {\bf \partial_sf(x,0)})
\end{array}
$$ is continuous.
In addition, we have 
$$\lambda_p(c\df+\M+ {\bf \partial_sf(x,0)})=\lambda_p(-c\df+\ms+ {\bf \partial_sf(x,0)})$$
and 
$$\lambda_p(c\df+\M+ {\bf \partial_sf(x,0)})=\lambda_p(-c\df+\ms+ {\bf \partial_sf(-x,0)})$$ 
where $\ms$ is the dual operator of $\M$ that is 
$$ \opms{\varphi}:=\int_{\R}J(y-x)\varphi(y)\,dy -\varphi(x).$$
\end{theorem}

With this two results, we can have a first  description of the effect of the speed $c$  on the persistence of the species. More precisely, we can show that

\begin{theorem}\label{cdm-thm2}
Let  $f,J$ be as in Theorem \ref{cdm-thm1} and assume that $\lambda_p(\M+{\bf \partial_sf(x,0)})<0$. Then there exists $0<c^{*,+}<c^{**,+}$ and $0<c^{*,-}<c^{**,-}$  such that
for all $-c^{*,-}<c<c^{*,+}$ there exists a positive solution to \eqref{cdm-eq} whereas none exists when $c\ge c^{**,+}$ or $c\le -c^{**,-}$.
\end{theorem}

The existence of such critical speed $c^{*,\pm}$ and $c^{**,\pm}$  comes as a corollary  of the sharp existence criteria obtained in Theorem \ref{cdm-thm1}  and the study of the behaviour of  $\lambda_p(c\df+\M+ {\bf \partial_sf(x,0)})$ with respect to the speed $c>0$, in particular its continuity and the equality  $\lambda_p(c\df+\M+ {\bf \partial_sf(x,0)})=\lambda_p(-c\df+\ms+ {\bf \partial_sf(x,0)})$ (Theorem \ref{cdm-thm3}). 

When  $J$ is symmetric, we have a more simple  description of the critical speed  $c^{*,+},c^{*,-},c^{**,+}$ and $c^{**,-}$. Indeed, in this situation we show that $c^{*,+}=c^{*,-}$ and $c^{**,+}=c^{**,-}$ and therefore we have  
\begin{theorem}\label{cdm-thm2-bis}
Let  $f,J$ be as in Theorem \ref{cdm-thm1} and assume further that $J$ is symmetric and  that $\lambda_p(\M+{\bf \partial_sf(x,0)})<0$. Then there exists $0<c^{*}<c^{**}$ such that
for all $|c|<c^{*}$ there exists a positive solution to \eqref{cdm-eq} whereas none exists when $|c|\ge c^{**}$.
\end{theorem}
This results is a direct consequence of Theorem \ref{cdm-thm3}, as for symmetric $J$, we then have $$\lambda_p(c\df+\M+ {\bf \partial_sf(x,0)})=\lambda_p(-c\df+\m+ {\bf \partial_sf(x,0)}).$$  
Note that due to possible non symmetrical structure of the nonlinearity  we might have expect that the critical speed $c^*,c^{**}$  depends in some ways  on the  sign of the speed of environmental change. This is surprisingly not the case,  meaning that there is no qualitative difference on the survival of the population when the niche  move to the right ($c>0$) or to the left ($c<0$), when the dispersal process is symmetric. In this situation  the critical speeds $c^*, c^{**}$ are  uniquely determined independently from the sign of the speed $c$. The existence of a preferred direction for survival  is uniquely determine by the symmetry of dispersal process.

The question of  sharp thresholds for the speeds, i.e. $c^*=c^{**}$, $c^{*,+}=c^{**,+}$ and $c^{*,-}=c^{**,-}$ as well as  the existence of formulas  describing them   as in the classical case  are still open problems which seems  intimately related to the properties of the kernel $J$. 
However, as in \cite{DeLeenheer2020}, we can  have a first estimate of $c^{**,\pm}$.  Namely we have

\begin{proposition}
Let  $f,J$ be as in Theorem \ref{cdm-thm1} and assume that $\lambda_p(\M+{\bf \partial_sf(x,0)})<0$. Then 
$$c^{**,+}\le c_\alpha^+:=\inf_{\alpha >0} \frac{1}{\alpha}\left(\int_{\R}J(z)e^{\alpha z}\,dz -1 +\sup_{x\in \R} f_{s}(x,0)\right)$$
$$c^{**,-}\le c_\alpha^-:=\inf_{\alpha >0} \frac{1}{\alpha}\left(\int_{\R}J(-z)e^{\alpha z}\,dz -1 +\sup_{x\in \R} f_{s}(x,0)\right)$$ 
\end{proposition}
   
Last we investigate the effect of the tail of the kernel on the survival of the population. To simplify our analysis we restrict our analysis to situations where $\sup_{\R}\partial_s f(x,0) >1$ and $J$ is symmetric and satisfies a second moment condition, that is 
\begin{equation}\label{hypj3}
\int_{\R}J(z)|z|^2\,dz<+\infty. \tag{H6}
\end{equation}

For such kernels and nonlinearities, we can extend the previous result of Theorem \ref{cdm-thm2}. Namely, we have 
\begin{theorem}\label{cdm-thm4}
Assume that $J,f$ satisfy \eqref{hypj1}-\eqref{hypf2}. Assume further that $\ds{\sup_{x\in \R}\partial_s f(x,0)>1}$ and that $J$ is symmetric and satisfies \eqref{hypj3}.  Then, there exists 
Then there exists $0<c^*<c^{**}$ such that
for all $|c|<c^*$ there exists a positive solution to \eqref{cdm-eq} whereas none exists when $|c|\ge c^{**}$.
\end{theorem}

Observe that the above results is a strong indication on  robustness of our results with respect to the tail of the kernel, in the sense that even for a fat-tailed kernel, for example $J(z)\sim \frac{1}{|z|^4}$, we observe the same pattern for the critical speed.  In particular, it highlights the fact that even population having a very strong spreading capabilities at a large scale can died out in a shifting environment.


\subsection{Further comments}
Before going to the proof of these results, we would like to make some further comments on the results obtained and explain our main strategy to prove them.

First, due to the particular structure of this problem, it is expected that  the properties of the principal eigenvalue $\lambda_p(c\df +\m +{\bf \partial_sf(x,0)})$ with respect to all its parameters will  play a key role in the proofs. This fact is also present in \cite{DeLeenheer2020} where we can see the important role played by the linearised operator around the trivial solution both  in the definition of the critical speed and  in the construction of a non trivial solution.  However, althougth we can find in \cite{DeLeenheer2020} some  elements related to the spectral theory of the  operator $c\df +\m +{\bf \partial_sf(x,0)}$ defined for particular structure of the zero order term, there is   no clear spectral theory  defined  for a general operators  $c\df +\m +{\bf \partial_sf(x,0)}$ whose zero order term is just assumed bounded.   Through the definition we give of  $\lambda_p(c\df +\m +{\bf a}) $  and our analysis of  its properties, we provide a clear  and acurate way for the description of the first fundamental element of the spectrum of these type of operators. \\ A great deal of the analysis presented here,  provides  a clear description of this spectral quantity and  establish some of its main properties. We believe that these fundamental results will be also of some help to understand other situation in particular in evolutionary biology \cite{Bouin2018a,Roques2020}.
\medskip
    
Second, we  would also  emphasize that  our result do not require any specific symmetry for the non-linearity as well for the kernel $J$,  properties that are strongly used in \cite{DeLeenheer2020} to derive their persistence criteria. Although our results mostly concerns compactly supported kernels, our  existence criteria  apply thereby on a larger variety of possible  non-linearity $f$ and $J$. 
In particular, we would like to highlight that as in the analysis of the case $c=0$  given in  \cite{Berestycki2016a,Brasseur2020}), our result on fat tailed kernel shows that the  assumption on support of the  kernel  $J$  we made is technical and   is not a prerequisite for analysing such type of problem. We found that the main difficulty in the analysis such situation relies on the lack of adequate tools, in particular the lack of an accurate spectral theory  for operator involving such kernels. Recent progress in the understanding of spectral properties of such operators  have been recently obtained in \cite{Cloez2020,Velleret2019} using probabilistic methods.
Wit this respect, with the proper spectral theory,  we believe that our results should hold true for a fractional version of \eqref{cdm-eq} where the operator $J\star \varphi -\varphi$ is replaced by the Fractional Laplacian $\Delta^s\varphi$.

Being  at the core the paper, let us briefly explain our main strategy to construct a positive non-trivial solution of \eqref{cdm-eq}. To construct such a nontrivial solution we  use the vanishing viscosity approach, approach previously used  in this context for example in \cite{Coville2007a,Coville2007d,Coville2008a}. The main idea  is  to introduce the following regularised problem : 
 
\begin{equation}\label{cdm-eq-eps-intro}
\eps\opddf{u}(x)+c\opdf{u}(x) +\opm{u}(x) +f(x,u(x))=0 \quad \text{ for } \quad x \in \R,
\end{equation} 
 and to show that for a fixed $c>0$ such that  $\lambda_p(c\df +\m +{\bf a})<0$  we can find $\eps_0(c)>0$ such that for $0<\eps\le \eps_0$, the problem \eqref{cdm-eq-eps} admits a unique positive non trivial solution. That is, to prove the following statement

\begin{theorem}\label{cdm-thm-eps-intro} 
 Assume that $f$ and $J$ satisfy $\eqref{hypj1}-\eqref{hypf2}$ and that  $c>0$ is such that $\lambda_p(c\df+\m+ \partial_s f (x,0)) < 0$. Then there exists $\eps_0>0$ such that for all $0<\eps \le \eps_0 $ there exists  a unique  positive smooth function, $u_\eps$, solution  to \eqref{cdm-eq-eps-intro}.
 \end{theorem}
 
From this result, we  can then obtain a solution  of  \eqref{cdm-eq}, by studying the singular limit problem when $\eps \to 0$. The proof of  Theorem \ref{cdm-thm-eps-intro} crucially relies on the properties of some spectral quantities.  We prove in particular that  for any $\O\subset \R$ domain we can define 
\begin{equation*}
\begin{split}
\lambda_p(\eps \ddf + c\df +\mb{\O} +{\bf a}):=\sup \{\lambda\in \R\,|\, &\exists \varphi \in C^0(\bar \O)\cap C^2(\O),\varphi>0\; \\ 
&\text{ such that }\; \eps\opddf{\varphi}+ c\opdf{\varphi}+\opmb{\varphi}{\O}+a(x)\varphi+\lambda\varphi\le 0\}.
\end{split}
\end{equation*}
and show that for any  compact domain $\O$  we have
$$\lim_{\eps \to 0}\lambda_p(\eps \ddf + c\df +\mb{\O} +{\bf a})=\lambda_p(c\df +\mb{\O} +{\bf a}).$$

 \medskip

The paper is organised as follows. In Section \ref{cdm-s-cp} we prove some useful comparison principle in unbounded domain. Then in Section \ref{cdm-s-mathbg}, we recall some basic theory on the principal eigenvalue $\lambda_p(c\df+\mb{\O}+{\bf a})$ and study several of its properties.  Within this section we analyse in particular the dependence of this principal eigenvalue with respect to the speed $c$ (Theorem \ref{cdm-thm3}) and prove Theorem \ref{cdm-thm2} assuming that the sharp existence criteria of  Theorem \ref{cdm-thm1} is true. In Section \ref{cdm-section-criteps} we obtained existence and uniqueness of the regularised problem \eqref{cdm-eq-eps-intro} and prove Theorem \ref{cdm-thm-eps-intro}. In Sections \ref{cdm-section-crit} \ref{cdm-section-uniq} and \ref{cdm-section-nonex}  we  prove the sharp existence criteria stated in Theorem \ref{cdm-thm1}, by proving successively, the sufficient condition for the  existence of  a  stationary solution (\ref{cdm-section-crit}),  the uniqueness of the stationary solution when it exists (\ref{cdm-section-uniq}), and  at last the necessary condition (\ref{cdm-section-nonex}). 
The analysis  of the long time behaviour is made in the last section, Section \ref{cdm-section-lgtb}, concluding the proof of Theorem \ref{cdm-thm1}. Finally, in the last section, Section \ref{cdm-s-fattailed}, we analyse the fat tailed case and prove Theorem \ref{cdm-thm4}.

\subsection{Notations}
To simplify the presentation we introduce some notations and various linear operator that we will use throughout this paper:
\begin{itemize}
\item We denote by $\mb{\O}$ the continuous linear operator 
\begin{equation}\label{cdm-def-opl}
\begin{array}{rccl}
\mb{\O}:&C(\bar \O)&\to& C(\bar \O)\\
&u&\mapsto& \int_{\O}J(x-y)u(y)\,dy -u,
\end{array}
\end{equation}
where $\O\subset \R$.
\item $\mb{R}$ corresponds to  the continuous operator $\mb{\O}$ with $\O=(-R,R)$, 
\item We will use  $\m$ to denote the operators $\mb{\O}$ with $\O=\R$ .
\item We will use also Euler's notation to denote differential operators namely  $\df$ and $\ddf$  will denote respectively the following differential operator $\opdf{f}=\frac{d}{dx}f(x)$ and $\opddf{f}:=\frac{d^2}{dx^2}f(x)$.    
\end{itemize}


\section{Comparison principles}\label{cdm-s-cp}
In this section we  collect  comparison principles that fit for our purposes. 
Let us start by proving a weak comparison principle that holds when  $a$ is negative outside a compact. Namely, we let us prove the following 

\begin{theorem}\label{cdm-thm-wcp}
Assume $\alpha,\beta, a\in C(\R)$ and $\alpha \ge 0$ (possibly $\equiv 0$) and let $J$ satisfies \eqref{hypj1} -- \eqref{hypj2}. Assume further there exists $R_0>0$ such that  $a(x)\le 0$ for all $|x|\ge R_0$. Let $u,v \in C(\R)\cap C^{2}_{loc}(\R)$ be such that 
\begin{equation*}
\begin{cases}
\alpha(x)\opddf{u}(x)+\beta(x)\opdf{u}(x)+\opm{u}(x)+a(x)u(x)\le 0 \quad \text{ for all }\quad |x|\ge R_0\\
\alpha(x)\opddf{v}(x)+\beta(x)\opdf{v}(x)+\opm{v}(x)+a(x)v(x)\ge 0 \quad \text{ for all }\quad |x|\ge R_0\\
\lim_{|x|\to +\infty}u(x)\ge \lim_{|x|\to +\infty}v(x)=0\\
v(x)\le u(x)  \quad \text{ for all }\quad |x|\le R_0
\end{cases}
\end{equation*} 
then $v(x)\le u(x)$ for all $x\in \R$.
\end{theorem}
\begin{proof}
The proof  is rather elementary but for the sake of completeness we include all the details. 
By definition of $a$  we can check that  for every $\delta>0$ we have 
\begin{align}
&\alpha(x) \opddf{u +\delta}(x) +\beta(x)\opdf{u+\delta}(x)+\opm{u+\delta}(x)+a(x)(u(x)+\delta) \,<\,0 & \text{ for all }  |x|>R_0, \label{cdm-eq-cp1} \\
&\alpha(x)\opddf{v}(x) +\beta(x)\opdf{v}(x)+\opm{v}(x)+a(x)v(x) \,\ge \,0 & \text{ for all }  |x|>R_0, \label{cdm-eq-cp2}\\
&v(x)<u(x)+\delta &\text{ for all }  |x|\le R_0.
\end{align}
In addition,  since $v\le u$ as $|x|\to +\infty$ there exists $R_\delta>R_0$ such that $v(x)< u(x)+\delta$ for all $|x|\ge R_\delta$. From  there  we then  see that $v\le u +\delta$ in $\R$. Indeed, if not  
then $\ds{\sup_{x_\in \R} (v(x)-u(x)-\delta) >0}$ and since $v<u+\delta$ in $[-R_0,R_0]\cup \left(\R\setminus (-R_\delta,R_\delta)\right)$ we have 
$$\sup_{\R}(v(x)-u(x)-\delta) =\max_{R_\delta <|x|<R_0} (v(x)-u(x)-\delta).$$
Let $x_0\in (-R_\delta,R_\delta)\setminus[-R_0,R_0]$ be the point where the function $ v-u -\delta$  achieved its maximum, then we have $\opdf{v-u-\delta}(x_0)=0$ and $\opddf{v -u -\delta}(x_0)\le 0, \opm{v-u-\delta}(x_0)\le 0, a(x_0)(v(x_0)-u(x_0)-\delta)\le 0$ and  
by evaluating \eqref{cdm-eq-cp1} and \eqref{cdm-eq-cp2} at this point we also get 
\begin{align*}
0&\le \alpha(x_0) \opddf{v -u -\delta}(x_0) +\beta(x_0)\opdf{v-u-\delta}(x_0)+\opm{v-u-\delta}(x_0)+a(x_0)(v(x_0)-u(x_0)-\delta)\\
&=\alpha(x_0) \opddf{v -u -\delta}(x_0) +\opm{v-u-\delta}(x_0)+a(x_0)(v(x_0)-u(x_0)-\delta) \le 0.
\end{align*}
Therefore  $\opm{v-u-\delta}(x_0)=0$ and we must have $v(x)-u(x)-\delta = v(x_0)-u(x_0)-\delta$ for all   $x \in x_0+ supp(J)$. 
By redoing the above argument with any point of $x_0+supp(J)$, we then see that $v(x)-u(x)-\delta = v(x_0) - u(x_0)-\delta$ for all  $ x \in x_0+2\cdot supp(J)$  and 
by arguing inductively we can then check that  for all $n\in \N,$  $v(x)-u(x)-\delta = v(x_0)- u(x_0)-\delta$ for all $x \in x_0+n \cdot supp(J)$. 
Now since $J$ satisfies \eqref{hypj2},i.e. $J(0)>0$,  then  $\lim_{n\to \infty}x_0 + n \cdot supp(J)=\R$ and thus we deduce the following contradiction 
$$ -\delta =\lim_{|x|\to +\infty} v(x)-u(x)-\delta= v(x_0) -u(x_0) -\delta >0.$$
Whence $v\le u+\delta$. 
The previous argumentation holding true for all $\delta\ge 0$, we then conclude that 
$$v(x)\le u(x) \text{ for all } \quad x\in \R.$$ 
\end{proof}

Last let us recall a classical strong comparison principle

\begin{theorem}\label{cdm-thm-scp}
Assume $\alpha,\beta, a\in C(\R)$ and $\alpha \ge 0$ (possibly $\equiv 0$) and let $J$ satisfies \eqref{hypj1} -- \eqref{hypj2}. Assume further that $a\in L^{\infty}$
and let $u,v \in C(\R)\cap C^{2}_{loc}(\R)$ be such that 
\begin{equation*}
\begin{cases}
\alpha(x)\opddf{u}(x)+\beta(x)\opdf{u}(x)+\opm{u}(x)+a(x)u(x)\le 0 \quad \text{ for all }\quad x\in \R\\
\alpha(x)\opddf{v}(x)+\beta(x)\opdf{v}(x)+\opm{v}(x)+a(x)v(x)\ge 0 \quad \text{ for all }\quad x\in \R\\
v(x)\le u(x)  \quad \text{ for all }\quad x\in  \R
\end{cases}
\end{equation*} 
then either  $v\equiv u$ or $v(x)<u(x)$ for all $x\in \R$.
\end{theorem}

\begin{proof}
The proof is rather standard. Since $a\in L^{\infty}$ there exists $k>0$ such that $a(x)-k\le 0$ and thus the non negative function $z:=u-v$ satisfies
$$
\alpha(x)\opddf{z}(x)+\beta(x)\opdf{z}(x)+\opm{z}(x)+a(x)z(x)\le 0 \quad \text{ for all }\quad x\in \R.$$
Now it is there standard to see that either $z\equiv 0$ or $z>0$. Indeed, if there exists $x_0$ such that $0=z(x_0)=\min_{x\in\R}z(x)$ then we get 
$$0\le \int_{\R}J(x_0-y)z(y)\,dy\le 0,$$
and by a classical argument we get $z\equiv 0$. So $v\equiv u$. Otherwise we have $z>0$ and so $u>v$.
  \end{proof}
  
\begin{remark}
Note that the above proofs only relies on the properties of the operator $\m$ and so the Theorems are still true if $\alpha(x)\equiv 0$ and/or $\beta \equiv 0$. In such situation, the $C^2_{loc}$ regularity for $u$ and $v$ is not needed and we can only require that $u,v \in C^1_{loc}(\R)$ or $C(\R)$ if also $\beta\equiv 0$. 
\end{remark}

\section{Spectral Theory of integro-differential operators}\label{cdm-s-mathbg}
 In this section, we recall some important results on  the principal eigenvalue of the linear non-local operator $c\df+\mb{\O}+{\bf a}$ and derive some new properties of this quantity, in particular its behaviour with respect to the parameter $c$.  We also recall some known variational characterisation  of the principal eigenvalue of general  integrodifferential operators. 
We split this section in two parts, one dealing with the spectral properties of the principal eigenvalue $\lambda_p$ defined for  nonlocal operators with a drift  and a second recalling some elements of the spectral theory regarding  integrodifferential operators containing an elliptic part. 
 
 \subsection{Principal eigenvalue for non-local operators with a drift}\label{cdm-ss-pge}
In this subsection, we recall some results on  the principal eigenvalue of a linear non-local operator $c\df+\mb{\O}+{\bf a}$ and establish new ones. That is,   we focus on the properties of the spectral problem 

\begin{equation}\label{cdm-eq-pev}
c\opdf{\varphi}+\opmb{\varphi}{\O}+a(x)\varphi+\lambda \varphi=0 \quad\text{ in }\quad \O.
\end{equation}

Following Berestycki, Nirenberg and Varadhan \cite{Berestycki1994}, we define the principal eigenvalue $\lambda_p$ the following way, 
$$\lambda_p=\sup\{\lambda\,|\, \exists \varphi \in C^1(\O),  \varphi>0,  c\opdf{\varphi}+\opmb{\varphi}{\O}+a(x)\varphi+\lambda \varphi\le 0\}.$$
When $c=0$, then $C^1(\O)$ the set of test functions is replaced by $C(\O)$. 
\subsubsection{Generic Properties}
 For such $\lambda_p$, let us  first recall some standard properties  that we constantly use throughout this paper:
 \begin{proposition}\label{cdm-prop-pev}
\begin{itemize}
\item[(i)] Assume $\O_1\subset\O_2$, then for the two operators $c\df+\mb{\O_1}+{\bf a}$ and $c\df+\mb{\O_2}+{\bf a}$
respectively defined on $C^1(\O_1)$ and $C^1(\O_2)$, we have :
 $$
\lambda_p(c\df+\mb{\O_1}+{\bf a})\ge \lambda_p(c\df+\mb{\O_2}+{\bf a}).
$$
\item[(ii)]For a fixed $\O$  and assume that $a_1(x)\ge a_2(x)$, for all $x \in \O$. Then  
$$
\lambda_p(c\df+\mb{\O}+{\bf a_2})\ge\lambda_p(c\df+\mb{\O}+{\bf a_1}).
$$

\item[(iii)] $\lambda_p(c\df+\mb{\O}+{\bf a})$ is Lipschitz continuous with respect to  $a$. More precisely,
$$|\lambda_p(c\df+\mb{\O}+{\bf a})- \lambda_p(c\df+\mb{\O}+{\bf b})|\le \|a-b\|_{\infty}$$
\item[(iv)] For any $\O\subset \R$, we always have the following bounds 
$$\lambda_p(c\df+\mb{\O}+{\bf a})\ge -\sup_{x\in\O}\left( \opmb{1}{\O}(x)+a(x)\right).$$
 \item[(v)]$\lambda_p(\mb{\O}+{\bf a})$ is continuous with respect to  $J$.
  \item[(vi)]$\lambda_p(c\df+\mb{\O}+{\bf a})$ is continuous with respect to  $c$.
\end{itemize}
\end{proposition}
The proofs are rather standard and essentially  use the definition of $\lambda_p$.  We point to \cite{Berestycki2016b,Coville2010,Coville2015} for the proofs of $(i)-(iv)$ in the situation where $c=0$ and to \cite{Coville2020} when $c\neq 0$.

Let us now state two important properties of the principal eigenvalue. 
The first one is a  Collatz-Wieland type characterization of $\lambda_p$. Namely,

\begin{theorem}\label{cdm-thm-CW}
Assume that $\O=(r_1,r_2)\subset \R$ (with $r_1<r_2$) is a bounded domain and let $a,J$ such that  $a\in C(\bar \O)\cap L^{\infty}(\O)$ and $J$ satisfies assumptions~\eqref{hypj1} -- \eqref{hypj2}.  Then 
$$\lambda_p(c\df+\mb{\O}+{\bf a})=\tilde \lambda_p'(c\df+\mb{\O}+{\bf a}),$$
where  for $c>0$
\begin{equation}\label{deftildelambdap'}\tilde \lambda_p'(c\df+\mb{\O}+{\bf a})=\inf\{\lambda\,|\, \exists \varphi \in C^1(\O)\cap C(\bar \O),  \varphi>0, \varphi(r_2)=0,  c\opdf{\varphi}+\opmb{\varphi}{\O}+a(x)\varphi+\lambda \varphi\ge 0\}.
\end{equation}
and for $c<0$
\begin{equation}\label{deftildelambdap'cneg}\tilde \lambda_p'(c\df+\mb{\O}+{\bf a})=\inf\{\lambda\,|\, \exists \varphi \in C^1(\O)\cap C(\bar \O),  \varphi>0, \varphi(r_1)=0,  c\opdf{\varphi}+\opmb{\varphi}{\O}+a(x)\varphi+\lambda \varphi\ge 0\}.
\end{equation}
In addition, there exists  a positive function $\varphi_1\in C^1(\O)\cap C(\overline{\O})$ such that 
$$\left\{\begin{array}{rcl}
c\df\varphi_1+\opmb{\varphi_1}{\O}+a\varphi_1 +\lambda_p \varphi_1 & = & 0\ \text{ in }\O=(r_1,r_2),\vspace{3pt}\\
\varphi_1 & > & 0\ \hbox{ in }\O=(r_1,r_2),\vspace{3pt}\\
\varphi_1(r_2) & = & 0\ \hbox{ if }c>0,\vspace{3pt}\\
\varphi_1(r_1) & = & 0\ \hbox{ if }c<0.\vspace{3pt}\\
\end{array}\right.$$
\end{theorem}

 The second property is a continuity result of the principal eigenvalue with respect to the domain and existence of a principal eigenfunction. Namely, we have
\begin{lemma}\label{cdm-lem-lim}
 Let $\O\subset \R$ be a domain and    let  $\mb{\O}$ be defined as in \eqref{cdm-def-opl} with  $J$ satisfying \eqref{hypj1} -- \eqref{hypj2}. Assume further that $J$ is compactly supported and $a\in C(\R)\cap L^{\infty}(\R)$.
Let $(\O_n)_{n \in \R}$ be an increasing sequence of bounded domain of $\R$ such that $\lim_{n\to \infty}\O_n =\O$, $\O_n \subset \O_{n+1}$.
Then,  we have 
$$ \lim_{n\to \infty}\lambda_p(c\df+\mb{\O_n}+{\bf a})=\lambda_p(c\df+\mb{\O}+{\bf a}).$$
In addition, there exists a positive smooth $\varphi_p$ associated with $\lambda_p$. 
\end{lemma} 

As above for the Proposition \ref{cdm-prop-pev}, the proof of these results are already contained or follow  from straightforward adaptation of the arguments developped in  \cite{Berestycki2016b, Coville2010} and \cite{Coville2020} and as such  we will omit their proofs.

\subsubsection{Behaviour of $\lambda_p$ with respect to the speed}

In this  part we study more precisely the behaviour of the principal eigenvalue with respect to the speed $c$. Let us first show some useful equalities.

\begin{proposition} \label{cdm-prop-lplp*}
Let $\O\subset \R$ be a domain,  $a\in C(\bar \O)\cap L^{\infty}(\O)$ and assume that $J$ compactly supported satisfies \eqref{hypj1} -- \eqref{hypj2}.  Then for all   $c$  we have
 \begin{equation}\label{cdm-eq-equality1}
 \lambda_p(c\df +\mb{\O} +{\bf a(x)})=\lambda_p(-c\df +\mbs{\O} +{\bf a(x)}).
 \end{equation}
  If in addition $\O$ is symmetric, in the sense that $\{-x\,|,x\in\O\}=\O$ then for all $c$ we have
\begin{equation} \label{cdm-eq-equality2}
 \lambda_p(c\df +\mb{\O} +{\bf a(x)})=\lambda_p(-c\df +\mbs{\O} +{\bf a(-x)}), 
 \end{equation}
 \end{proposition}  

\begin{proof}
Without any loss of generality we may assume that $c>0$.
First, let us assume that $\O$ is bounded set, that is $\O:=(r_1,r_2)$. By Theorem \ref{cdm-thm-CW}, there exists  $\varphi,\varphi^* \in C(\bar \O)\cap C^{1}(\O)$  positive eigenfunctions associated respectively with $\lambda_p(c\df +\mb{\O} +{\bf a})$ and $\lambda_p(-c\df +\mbs{\O} +{\bf a})$. Moreover they satisfies , $\varphi^*(r_1)=\varphi(r_2)=0$ and 
\begin{align*}
&c\df\varphi+\opmb{\varphi}{\O}+a\varphi +\lambda_p(c\df +\mb{\O} +{\bf a}) \varphi  =  0\\
-&c\df\varphi^*+\opmbs{\varphi^*}{\O}+a\varphi^* +\lambda_p(-c\df +\mbs{\O} +{\bf a})\varphi^*  =  0
\end{align*}
Let us multiply by $\varphi^*$ the equation satisfied by $\varphi$ and integrate the resulting equation over $\O$. Integrating by parts and using the equation satisfied by $\varphi^*$ since $\varphi^*(r_1)=\varphi(r_2)=0$ it follows that 
$$(\lambda_p(c\df +\mb{\O} +{\bf a})-\lambda_p(-c\df +\mbs{\O} +{\bf a})) \int_{\O}\varphi\varphi^* =0. $$
Thus the equality $\ds{\lambda_p(c\df +\mb{\O} +{\bf a})=\lambda_p(-c\df +\mbs{\O} +{\bf a})}$ holds true since $\varphi\varphi^*>0$.
By continuity of $\lambda_p(c\df +\mb{\O} +{\bf a})$ with respect to the domain (Lemma \ref{cdm-lem-lim}) the equality \eqref{cdm-eq-equality1} then holds true for any domain $\O$. 
The second equality is obtained just by observing that if $(\varphi,\lambda)$ satisfies 
 $$c\df\varphi+\opmb{\varphi}{\O}+a\varphi +\lambda_p(c\df +\mb{\O} +{\bf a}) \varphi  =  0,$$
 then since $\O$ is symmetric the function $\psi(x):=\varphi(-x)$ satisfies 
 $$-c\df\psi+\opmbs{\psi}{\O}+a(-x)\psi +\lambda_p(c\df +\mb{\O} +{\bf a}) \psi  =  0.$$
 The equality \eqref{cdm-eq-equality2} then follows by using the definition of the principal eigenvalue.
\end{proof}

 Next, we prove the following  elementary property. 
\begin{proposition} \label{cdm-prop1-behave-c}
 Assume  that $a\in C(\R)\cap L^{\infty}(\R)$   and that $J$  compactly supported satisfies \eqref{hypj1} -- \eqref{hypj2}. Assume further that $\sup_{\R}a(x)>0$.
 Then there exists  $c_0^{+}>0$  such that for all $c> c_{0}^+,$ $$\lambda_p(c\df +\m +{\bf a})> 0.$$
Similarly there exists $c_{0}^{-}>0$ such that 
 for all $c\le -c_{0}^-$, $$\lambda_p(-c\df +\ms +{\bf a})> 0. $$
 \end{proposition}  

\begin{proof}
We  treat the two situation $c> 0$ and $c<0$ separately.
\subparagraph{Case  $c> 0$: } In this situation, take  $\varphi(x):=e^{-\lambda x}$, then  we have
$$c\opdf{\varphi} +\opm{\varphi}+a(x)\varphi =e^{-\lambda x}\left(-c\lambda+\int_{\R}J(z)e^{\lambda z}\,dz-1+a(x)\right).$$
Define $$c_0^{+}:=\inf_{\lambda>0} \frac{1}{\lambda}\left(\int_{\R}J(z)e^{\lambda z}\,dz-1+\sup_{x\in \R}a(x)\right).$$
$c_0^+$ is  bounded quantity  since by an elementary computation we get   
$$\frac{1}{\lambda}\left(\int_{\R}J(z)e^{\lambda z}\,dz-1+\sup_{x\in \R}a(x)\right)\ge \int_{\R}J(z)z\,dz.$$
By definition of  $c_0^+$, for all $c> c_0^+$,  there exists $\lambda(c)$ such that  
$$h(\lambda(c)):=-c\lambda(c)+\int_{\R}J(z)e^{\lambda(c) z}\,dz-1+\sup_{x\in \R} a(x)<0.$$ 
Whence  for such $\lambda(c)$ we have
\begin{align*}
c\opdf{\varphi} +\opm{\varphi}+a(x)\varphi &=\varphi\left(-c\lambda+\int_{\R}J(z)e^{\lambda z}\,dz-1+\sup_{x\in \R} a(x)\right) + (a(x)-\sup_{x\in \R} a(x))e^{-\lambda(c)x} \\
&\le h(\lambda(c)) \varphi<0.
\end{align*}
By definition of the principal eigenvalue of the operator $c\df+\m +{\bf a}$,   $(\varphi,-h(\lambda(c))$ is then an admissible test function and therefore $ \lambda_p(c\df +\m +{\bf a})\ge -h(\lambda(c))> 0$.

\subparagraph{Case $c<0$: }In this situation, again take  $\varphi(x):=e^{-\lambda x}$ and observe that 
$$-c\opdf{\varphi} +\opms{\varphi}+a(x)\varphi =\varphi\left(c\lambda+\int_{\R}J(-z)e^{\lambda z}\,dz-1+a(x)\right).$$
Define
 $$c_0^{-}:=\inf_{\lambda>0} \frac{1}{\lambda}\left(\int_{\R}J(-z)e^{\lambda z}\,dz-1+\sup_{x\in \R}a(x)\right).$$
Again $c_0^-$ is a bounded quantity, since a similar  elementary computation shows that  
$$\frac{1}{\lambda}\left(\int_{\R}J(-z)e^{\lambda z}\,dz-1+\sup_{x\in \R}a(x)\right)\ge \int_{\R}J(-z)z\,dz.$$
As above we can construct a admissible test function for $c<-c_0^-$. Indeed, when $c<-c_0^-$, then  $-c>c_0^-$ and  we can find $\lambda(c)>0$  such that 
$$\delta(c):=\left(c\lambda(c)+\int_{\R}J(-z)e^{\lambda(c) z}\,dz-1+\sup_{x\in \R} a(x)\right)<0.$$ 
Therefore $(\varphi,\delta(c))$ satisfies
$$-c\opdf{\varphi} +\opms{\varphi}+a(x)\varphi -\delta(c)\varphi = (a(x)-\sup_{x\in \R} a(x))\varphi\le 0. $$ 
The couple $(\varphi,-\delta(c))$ is therefore a admissible test function for the principal eigenvalue of the operator $-c\df+\ms+{\bf a(x)}$ and as such
we have $ \lambda_p(-c\df +\ms +{\bf a})\ge -\delta(c)> 0.$
By using Proposition \ref{cdm-prop-lplp*}, we then  conclude  
$$\lambda_p(c\df +\m +{\bf a})=\lambda_p(-c\df +\ms +{\bf a})\ge -\delta(c)> 0.$$
\end{proof}

\begin{remark}
Note that when $J$ is symmetric then $c^+_0=c^-_0=:c_0$ and  we have the following quantitative bound for $c_0:$
$$c_0\ge \inf_{\lambda>0} \left(\lambda\int_{\R}J(z)z^2\,dz +\sup_{x\in\R} a(x)\frac{1}{\lambda}\right)=2\sqrt{\int_{\R}J(z)z^2\,dz \cdot\sup_{x\in \R} a(x)}.$$
  \end{remark}

Let us now show that for $c$ close to $0$ then  $\lambda_p(c\df+\m+\bf{a})<0$. To do so let us first show it for $c>0$, namely 
\begin{lemma}\label{cdm-lem-c+}
 Assume  that $a\in C^{0,\alpha}(\R)\cap L^{\infty}(\R)$ and $J$ with compact support satisfies \eqref{hypj1}-\eqref{hypj2}. Assume further that $\lambda_p(\m+{\bf a})<0$ then there exists  $c^{*,+}>0$ such that  for all $0<c<c^{*,+},\, \lambda_p(c\df +\m +{\bf a})<0$.
 \end{lemma}

 \begin{proof}
To simplify the proof let us denote  $\lambda_0:=\lambda_p(\m+{\bf a})<0$ and without any loss of generality let us assume that $\lambda_p(\m+{\bf a})$ is associated with a positive principal eigenfunction $\varphi_p$ otherwise using the Lipschitz continuity of $\lambda_p(\M +{\bf a})$ with respect to $a$, we can perturb $a$ by $a_\eps$ with $a_\eps$ such that 
 $\lambda_p(\m+{\bf a_\eps})<\frac{\lambda_0}{2}$ and there exists $\varphi_p$ associated with $\lambda_p(\m+{\bf a_\eps})$.  

Now thanks to the continuity   of $\lambda_p(\mb{\O}+{\bf a})$ with respect to the domain (Lemma \ref{cdm-lem-lim} ), we have for some $R_0$, $\lambda_p(\mb{R_0}+{\bf a})<\frac{\lambda_0}{2}$ and there exists $\varphi_p$ associated with $\lambda_p(\mb{R_0}+{\bf a}) $. 
Let us now introduce $\mb{\gamma,R_0}$ the following operator defined for $\varphi \in C([-R_0,R_0])$ by 
$$ \opmb{\varphi}{\gamma,R_0}(x):=(1-\gamma)\int_{-R_0}^{R_0}J(x-y)\varphi(y)\,dy -\varphi(x).$$
Thanks to the continuity of $\lambda_p$ with respect to $a(x)$, for $\gamma>0$ small says, $\gamma\le \gamma_0$ we have 
$$ \lambda_p(\mb{\gamma,R_0}+{\bf a})\le \frac{\lambda_0}{4}. $$

Next, we claim 
\begin{claim}
For  any $\delta>0$, there exists $\psi\in C^1((-R_0,R_0))\cap C_c([-R_0,R_0])$, $\psi\ge 0$ such that
\begin{align*} 
&\opmb{\psi}{\gamma,R_0}(x)+ (a(x)+ \lambda_p(\mb{\gamma,R_0}+{\bf a}) +2\delta)\psi(x)\ge 0\qquad \text{for all}\quad x\in (-R_0,R_0)\\  
&\inf_{x\in (-R_0,R_0)}\int_{-R_0}^{R_0}J(x-y)\psi(y)\,dy>0
\end{align*}
\end{claim}
Assume for the moment that the claim holds true, then we can  infer that 
\begin{equation}\label{cdm-eq-limsupc}
\limsup_{c\to 0, c>0} \lambda_p(c\df+ \mb{R_0}+{\bf a})\le \lambda_p(\mb{\gamma,R_0}+{\bf a}) +2\delta.
\end{equation}
Indeed, from the claim a direct computation shows that  
$$c\opdf{\psi} +\opmb{\psi}{R_0}+(a(x)+ \lambda_p(\mb{\gamma,R_0}+{\bf a}) +2\delta)\psi(x)\ge c\opdf{\psi} +\gamma d_0$$
where 
$\ds{d_0:=\inf_{x\in(-R_0,R_0)}\int_{-R_0}^{R_0}J(x-y)\psi(y)\,dy>0}$.
Therefore,  for  $c>0$ small enough we achieve
\begin{equation}\label{cdm-eq-psipos}
c\opdf{\psi} +\opmb{\psi}{R_0}+(a(x)+ \lambda_p(\mb{\gamma,R_0}+{\bf a}) +2\delta)\psi(x)\ge c\opdf{\psi} +\gamma d_0>0.
\end{equation}
As a consequence, for $c>0$ small we then have 

$$\tilde\lambda'_p(c\df+ \mb{R_0}+{\bf a})\le \lambda_p(\mb{\gamma,R_0}+a) +2\delta. $$

Now thanks to the Collatz Wieland type characterisation of $\lambda_p(c\df+ \mb{R_0}+{\bf a})$,  Theorem \ref{cdm-thm-CW}, from the above inequality 
 we deduce that for $c>0$ small
$$ \lambda_p(c\df+ \mb{R_0}+{\bf a})\le \lambda_p(\mb{\gamma,R_0}+{\bf a}) +2\delta,$$
which enforces \eqref{cdm-eq-limsupc}.
Now by choosing $\delta$ small enough, we then get 
$$ \limsup_{c\to 0, c>0} \lambda_p(c\df+ \mb{R_0}+{\bf a})< \frac{\lambda_0}{8}.$$ 

As a consequence there exists $c^{*,+}>0$, such that for all $0<c\le c^{*,+}$ $\lambda_p(c\df+ \mb{R_0}+{\bf a})\le \frac{\lambda_0}{16}$  and thanks to the monotone behaviour of $\lambda_p(c\df+ \mb{R_0}+{\bf a})$ with respect to the domain, we have 
$$\lambda_p(c\df+ \m+{\bf a})\le \lambda_p(c\df+ \mb{R_0}+{\bf a}) \le \frac{\lambda_0}{16}<0.$$
\end{proof}
 
In order to  conclude the proof  let us now establish the claim. 
\begin{proof}[Proof of the Claim]
Let $\delta>0$ be a fixed. Arguing as  in Claim 3.2 of \cite{Berestycki2016b},  we can infer that there exists $\eps_0>0$ such that for any $0<\eps\le \eps_0$ then there exists $\varphi_\eps \in C_c((-R_0,R_0))$ verifying 
\begin{align*}
&\opmb{\varphi_\eps}{\gamma,R_0}(x)+(a(x)+\lambda_p(\mb{\gamma,R_0}+{\bf a})+\delta)\varphi_\eps(x) \ge 0 \quad \text{ for }\quad x\in (-R_0,R_0).\\
& (-R_0+\eps,R_0-\eps)\subset supp(\varphi_\eps). 
\end{align*}
Since $J$ satisfies $(H1-H2)$ we can fix now $\eps$ small, such that $$\inf_{x\in \O}\int_{-R_0}^{R_0}J(x-y)\varphi_\eps(y)\,dy>0.$$ 

Observe that by construction, since $\varphi_\eps \in C_c((-R_0,R_0))$, we can easily check that 
$$\opmb{\varphi_\eps}{\gamma,\R}(x) +(a(x)+\lambda_p(\mb{\gamma,R_0}+{\bf a})+\delta)\varphi_\eps(x)\ge 0 \quad \text{ for all }\quad x \in \R.$$

Now, let $\zeta$ be a smooth  mollifier of unit mass and with support in the unit interval and consider  $\zeta_\tau:=\frac{1}{\tau}\zeta\left(\frac{z}{\tau}\right)$ for $\tau>0$. 

 By taking $\psi:= \zeta_\tau\star \varphi_\eps$ and  observing that $\opmb{\psi}{\gamma,\R}(x)=\zeta_{\tau}\star(\opmb{\varphi_\eps}{\gamma,\R})(x) $ for any $x\in \R$, we deduce that 
   \begin{align*}
   &\zeta_\tau \star\left(\opmb{\varphi_\eps}{\gamma,\R} +(a(x)+\lambda_p(\mb{\gamma,R_0}+{\bf a})+\delta)\varphi_\eps\right)\ge 0 \quad \text{ for all }\quad x \in \R, \\
   & \opmb{\psi}{\gamma,\R}(x) +(\lambda_p(\mb{\gamma,R_0}+{\bf a})+\delta)\psi(x) +\zeta_\tau\star (a(x)\varphi_\eps)(x)\ge 0 \quad \text{ for all }\quad x \in \R.
   \end{align*}
  By adding and subtracting $a$,  we then have  for all $x\in \R$
  $$
    \opmb{\psi}{\gamma,\R}(x) +(a(x)+\lambda_p(\mb{\gamma,R_0}+{\bf a}) +\delta)\psi(x) + \int_{\R}\zeta_\tau(x-y)\varphi_\eps(y)(a(y)-a(x))\,dy\ge 0.
   $$
  For $\tau$ small enough, say $\tau \le \tau_0$, the function   $\psi\in C^{\infty}_{c}((-R_0,R_0))$ and for all $x\in (-R_0,R_0)$ we have 
 \begin{align*}
 \opmb{\psi}{\gamma,\R}(x)&=(1-\gamma)\int_{\R}J(x-y)\psi(y)\,dy-\psi(x),\\
 &=(1-\gamma)\int_{-R_0}^{R_0}J(x-y)\psi(y)\,dy-\psi(x)= \opmb{\psi}{\gamma,R_0}(x).
 \end{align*}  
   Thus, from the above inequalities, for $\tau \le \tau_0$, we get for all $ x \in (-R_0,R_0),$
   $$  \opmb{\psi}{\gamma,R_0}(x) +(a(x)+\lambda_p(\mb{\gamma,R_0}+{\bf a})+\delta)\psi(x) + \int_{\R}\zeta_\tau(x-y)\varphi_\eps(y)(a(y)-a(x))\,dy\ge 0.$$
   Since $a$ is H\"older continuous, we can estimate the integral  by  
   \begin{align*}
   \left|\int_{\R}\zeta_\tau(x-y)\varphi_\eps(y)(a(y)-a(x))\,dy \right|&\le \int_{\R}\zeta_\tau(x-y)\varphi_\eps(y)\left|\frac{a(y)-a(x)}{|y-x|^{\alpha}}\right||x-y|^{\alpha}\,dy,\\
&\le \kappa\tau^{\alpha}\psi(x),
   \end{align*}
   where $\kappa$ is the H\"older semi-norm of $a$.
   Thus, for $\tau$  small, says $\tau\le \inf\{\left(\frac{\delta}{2 \kappa}\right)^{1/\alpha},\tau_0\}$, we have
     \begin{equation*}
   \opmb{\psi}{\gamma,R_0}(x) +(a(x)+\lambda_p(\mb{\gamma,R_0}+{\bf a})+2\delta)\psi(x)\ge 0 \quad \text{ for all } \quad x\in (-R_0,R_0),
   \end{equation*}
   In addition, since $supp(\varphi_\eps)\subset supp(\psi)$, we then infer that $$\inf_{x\in (-R_0,R_0)}\int_{-R_0}^{R_0}J(x-y)\psi(y)\,dy>0,$$
   which end the proof of the claim. 
\end{proof} 

Last let us prove that $\lambda_p(c\df+\m+\bf{a})<0$ for $c<0$ and  $c$ close to $0$. Namely,
\begin{lemma}\label{cdm-lem-c-}
 Assume  that $a\in C^{0,\alpha}(\R)\cap L^{\infty}(\R)$ and $J$ with compact support satisfies \eqref{hypj1}-\eqref{hypj2}. Assume further that $\lambda_p(\m+{\bf a})<0$ then there exists  $c^{*,-}>0$ such that  for all $-c^{*,-}<c<0, \; \lambda_p(c\df +\m +{\bf a})<0$.
 \end{lemma}
\begin{proof}
To obtain $c^{*,-}$, let us observe that thanks to Proposition \ref{cdm-prop-lplp*} we have 
\begin{align*}
&\lambda_p(c\df +\m+{\bf a})=\lambda_p(-c\df +\ms +{\bf a}),\\
&\lambda_p(\m+{\bf a})=\lambda_p(\ms+{\bf a}).
\end{align*}
Let us then consider the operator $-c\df + \ms+{\bf a}$. Since $\lambda_p(\ms+{\bf a})<0$ and  $-c>0$ and we 
 can therefore apply the Lemma \ref{cdm-lem-c+} to  $-c\df +\ms +{\bf a}$. As a consequence there exists $\bar c$ such that for all 
$-c\le \bar c $ we have $\lambda_p(-c\df +\ms+{\bf a})<0$. 
 Hence, by denoting $c^{*,-}:=\bar c$ and by using the above equality, for all $c>-\bar c$, we have  
 $$\lambda_p(c\df +\m+{\bf a})=\lambda_p(-c\df +\ms+{\bf a})<0.$$
\end{proof}

By Propositions \ref{cdm-prop1-behave-c} and \ref{cdm-prop-lplp*}, Lemma \ref{cdm-lem-c+} and Lemma \ref{cdm-lem-c-} we can then define  
\begin{align}
&c^{*,+}:=\sup\{c>0\, |\, \forall\, 0<c'\le c,\lambda_p(c'\df+\m+{\bf a})<0\},\label{cdm-eq-def-c*+}\\
&c^{*,-}:=\inf\{c<0\, |\, \forall\, c\le c'<0,\lambda_p(c'\df+\m+{\bf a})<0\},\label{cdm-eq-def-c*-}
\end{align}

and 
\begin{align}
c^{**,+}:=\inf\{c>0 \, |\, \forall\, c'\ge c, \lambda_p(c'\df+\m+{\bf a})\ge 0\},\label{cdm-eq-def-c**+}\\
c^{**,-}:=\sup\{c<0 \, |\, \forall\, c'\le c, \lambda_p(c'\df+\m+{\bf a})\ge 0\}.\label{cdm-eq-def-c**-}
\end{align} 

Observe that from the definition of $c^{*,\pm},c^{**,\pm}$ and thanks to Proposition \ref{cdm-prop-lplp*}, the Theorem \ref{cdm-thm2} will then be proved as soon as the optimal persistence criteria (Theorem \ref{cdm-thm1}) is proved.

\subsection{Principal eigenvalue for nonlocal operators with an elliptic part}\label{cdm-ss-pge-eps}

In this section we recall the definition and behaviour of the principal eigenvalue for general integrodifferential operator $\lg$ of the form
$$\oplg{\varphi}:=\ope{\varphi} +\oplb{\varphi}{\O}$$
where $\e$ is an elliptic operator of the form $$\ope{\varphi}:=\alpha(x)\opddf{\varphi} +\beta(x)\opdf{\varphi} +\gamma(x)\varphi$$
where $\alpha(x)>0$. 
As in \cite{Berestycki1994,Berestycki2015,Berestycki2016b,Coville2017a,Coville2020, Nussbaum1992},  we can check that the quantity 
$$\lambda_p(\lg):=\sup\{\lambda\in\R\,|\,\exists\, \varphi\in C(\bar \O)\cap W^{2,1}_{loc}(\O),\, \varphi>0,\, \oplg{\varphi}+\lambda \varphi\le 0   \} $$ 
is well defined and satisfies all the properties defined in Proposition \ref{cdm-prop-pev}.
Moreover we also have the following Collatz-Wieland  characterisation 
\begin{theorem}[\cite{Coville2017a}]\label{cdm-thm-ellip}
Assume that $\O=(r_1,r_2)\subset \R$ (with $r_1<r_2$) is a bounded domain and let  $\alpha,\beta,\gamma\in C^{0,\alpha}(\O)\cap C(\bar \O)$ and $J$ with compact support satisfying \eqref{hypj1}-\eqref{hypj2}. Assume further that $\alpha>0$, then $\lambda_p(\lg)$ is the principal eigenvalue of $\lg$, meaning that there exists a positive smooth function $\varphi_p$ associated to $\lambda_p(\lg)$, such that 
\begin{align*}
&\oplg{\varphi_p}(x)+\lambda_p \varphi_p(x)=0 \quad \text{ for all } \quad x\in \O,\\
&\varphi_p(r_1)=\varphi_p(r_2)=0.
\end{align*}
Moreover,  $\lambda_p(\lg)$ satisfies the following Collatz-Wieland characterisation:
$$
\lambda_p(\lg)=\lambda_p'(\lg):=\inf\{\lambda\in\R\,|\,\exists\, \varphi\in C^2(\O)\cap C(\bar \O),\, \varphi>0, \varphi(r_1)=\varphi(r_2)=0, \, \oplg{\varphi}+\lambda \varphi\ge 0   \}.
$$
\end{theorem}
 

\section{Analysis of the regularised problem}\label{cdm-section-criteps}

In this section we construct a positive  solution $u_\eps$ of a regularised version of \eqref{cdm-eq}.
For $\eps >0$ let us introduce the following regularised problem : 
 
\begin{equation}\label{cdm-eq-eps}
\eps\opddf{u}(x)+c\opdf{u}(x) +\opm{u}(x) +f(x,u(x))=0 \quad \text{ for all } \quad x \in \R.
\end{equation} 
 
 Then, for a fixed $c>0$ such that  $\lambda_p(c\df +\m +{\bf a})<0$, we will show that  we can find $\eps_0(c)$ such that for $\eps\le \eps_0$, the problem \eqref{cdm-eq-eps} admits a unique positive non trivial solution. Namely, we prove the following
 
  \begin{theorem}\label{cdm-thm-eps} 
 Assume that $f$ and $J$ satisfy $\eqref{hypj1}-\eqref{hypf2}$ and that  $c>0$ is such that $\lambda_p(c\df+\m+ \partial_s f (x,0)) < 0$. Then there exists $\eps_0>0$ such that for all $0<\eps \le \eps_0 $ there exists  a unique  positive continuous function, $u_\eps$, solution  of \eqref{cdm-eq-eps}.
 \end{theorem}
 
The next subsections deal with the proof of Theorem \ref{cdm-thm-eps}. 
The proof of the Theorem \ref{cdm-thm-eps}  uses standard approximation schemes that we can find for example in  \cite{Coville2007d,Berestycki1990, Berestycki1992}. To simplify the presentation we break  down  the proof in three parts, the next two subsection being devoted to each one of them. In the first part, subsection \ref{cdm-ss-constr1},  we introduce the following approximated problem :
\begin{align}
&\eps \opddf{u}(x) + c\opdf{u}(x) +\opmb{u}{R}(x) +f(x,u(x))=0\qquad \text{ for all }\qquad  x\in (-R,R) \label{cdm-eq-reg-kpp}\\
&u(-R)=u(R)=0\label{cdm-eq-reg-kpp-bc}
\end{align}
and we find $\eps_0$  and $R_0$ positive constants such that for all $R\ge R_0$ and $\eps\le \eps_0$ the problem \eqref{cdm-eq-reg-kpp}-\eqref{cdm-eq-reg-kpp-bc} set on a  $(-R,R)$  has a unique solution. More precisely, we prove 
\begin{theorem}\label{cdm-thm-eps-bd} 
 Assume that $f$ and $J$ satisfy $\eqref{hypj1}-\eqref{hypf2}$ and that  $c>0$ is such that $\lambda_p(c\df+\m+ \partial_s f (x,0)) < 0$. Then there exists $R_0>0$ and $\eps_0>0$ such that for all $0<\eps \le \eps_0 $, $R\ge R_0$  there exists  a unique positive continuous function, $u_{\eps, R}$, solution  to \eqref{cdm-eq-reg-kpp}--\eqref{cdm-eq-reg-kpp-bc}.
Moreover, for any smooth initial data $v_0\in C^1((-R,R)), v_0\ge_{\not \equiv} 0$ then the solution $v(t,x)$ of the following Cauchy problems
 \begin{align*}
&\partial_t v(t,x)=\eps \opddf{v}(t,x) + c\opdf{v}(t,x) +\opmb{v}{R}(t,x) +f(x,v(t,x))\qquad \text{ for }\qquad t>0, x\in (-R,R)\\
&v(t,-R)=v(t,R)=0\qquad \text{ for } \qquad t>0\\
&v(0,x)=v_0(x)\qquad \text{ for } \qquad x\in (-R,R),
\end{align*}
 converges uniformly to $u_{\eps,R}$.
 \end{theorem}

Then in a second part,subsection \ref{cdm-ss-constr2}, by using a standard limiting procedure and  a-priori estimates, we show that $\eps_0$ is independent of $R$ and that for any $0<\eps \le \eps_0$ a positive non trivial solution to \eqref{cdm-eq-eps} can be constructed from  the sequences $(u_{\eps,R_n})_{n\in\N}$  with  $R_n\to \infty$. In the last part, subsection \ref{cdm-ss-constr3} we prove the uniqueness of the solution of \eqref{cdm-eq-eps}.

\subsection{Existence of a unique non-trivial positive solution in a bounded domain}\label{cdm-ss-constr1}

We start this subsection by proving  the uniqueness of the positive solution of \eqref{cdm-eq-reg-kpp}--\eqref{cdm-eq-reg-kpp-bc}. Our proof follows a standard argument that we can find for example in \cite{Berestycki2005,Coville2013}. Let $u$ and $v$ be two non negative solutions to  \eqref{cdm-eq-reg-kpp}--\eqref{cdm-eq-reg-kpp-bc}. Then thanks to the maximum principle and using the equation satisfied by $u$ and $v$ there exists $M>0$ such that  $M>u>0, M>v>0$ and moreover we have
\begin{align*}
\opdf{u}(-R)>0,\; \opdf{v}(-R)>0\\
\opdf{u}(R)<0,\; \opdf{v}(R)<0.
\end{align*}
So there exists $\sigma_0>0$ such that $\frac{1}{\sigma_0}v\le u\le \sigma_0 v$. From this inequalities we can define  
$$\sigma^*:= \inf\{\tau>0\, |\, \sigma v\}$$
and by definition of $\sigma^*$, we have $u\le \sigma^* v$. We claim that
\begin{claim}
$\sigma^*\le 1.$
\end{claim} 
 Note that by proving the claim we  deduce that  $u\le v$ and since  the role of $u$ and $v$  can be interchanged we then obtain $v\le u$ as well,  which proves the uniqueness.  
 
 \begin{proof}[Proof of the Claim:\nopunct]
 Assume by contradiction that $\sigma^*>1$. Then by using that $v$ is a solution, we can check that $\sigma^*v$ is a super-solution to \eqref{cdm-eq-reg-kpp}--\eqref{cdm-eq-reg-kpp-bc}. Now since $u\le \sigma^*v$ and $f$ is locally Lipschitz thanks to the strong maximum principle, we can check that either  $u<\sigma^* v$ or $u\equiv \sigma^* v $. In the latter case, since $\frac{f(x,s)}{s}$ is strictly decreasing we then have the following contradiction 
 \begin{align*}
 0=\eps \opddf{u}+ c \opdf{u}+ \opmb{u}{R} +\frac{f(x,u)}{u}u  &= \eps \opddf{\sigma^*v}+ c \opdf{\sigma^*v}+ \opmb{\sigma^* v}{R} +\frac{f(x,\sigma^*v)}{\sigma^* v}\sigma^*v\\
                   &<\sigma^*\left( \eps \opddf{v}+ c \opdf{v}+ \opmb{v}{R} +f(x,v)\right)=0.
 \end{align*}
 In the other case, we also get a contradiction. Namely, since $u<\sigma^*v$ and  thanks to the equations satisfied by $u$ and $\sigma^*v$ we can see that $\opdf{u}(-R)<\sigma^*\opdf{v}(-R) $ and $\opdf{u}(R)
>\sigma^*\opdf{v}(R)$. Therefore $u\le (\sigma^*-\eps)v$ for  some $\eps$ small contradicting the definition of $\sigma^*$. 
\end{proof}

 Let us prove now  that for $\eps$ small and $R$ well chosen,   a positive solution to \eqref{cdm-eq-reg-kpp} --\eqref{cdm-eq-reg-kpp-bc} exists. For convienience, we use the following notation $a(x):=\partial_sf(x,0)$.
  Note that the problem \eqref{cdm-eq-reg-kpp}--\eqref{cdm-eq-reg-kpp-bc} can be solved using standard sub and super-solution scheme (see \cite{Coville2007d}) and that large constants are super-solutions of this problem. So, in order to construct a solution  it is enough to construct a bounded sub-solution of  \eqref{cdm-eq-reg-kpp} -- \eqref{cdm-eq-reg-kpp-bc} and thanks to the nature of the problem \eqref{cdm-eq-reg-kpp}--\eqref{cdm-eq-reg-kpp-bc}, it is enough to show that 
\begin{equation}
\lambda_p(\eps\ddf +c\df +\mb{R} +{\bf a})<0. \label{cdm-eq-lpeps}
\end{equation}
Indeed, let assume that $\lambda_p(\eps\ddf +c\df +\mb{R} +{\bf a})<0$ and let $\varphi_p$ the associated positive eigenfunction. Then a straightforward computation shows that for small $\kappa$ the function  $\kappa\varphi_p$ satisfies $\kappa \varphi_p(R)=\kappa\varphi_p(-R)=0$ and  
$$
\eps \opddf{\kappa\varphi_p}(x) + c\opdf{\kappa\varphi_p}(x)+\opmb{\kappa \varphi_p}{R}(x) +f(x,\kappa \varphi_p(x))=-\lambda_p\kappa \varphi_p(x) + o(\kappa \varphi_p(x))\ge 0.
$$
As a consequence, for small $\kappa$ the function $\kappa\varphi_p$ is a subsolution of \eqref{cdm-eq-reg-kpp}--\eqref{cdm-eq-reg-kpp-bc}. 

\noindent Let us now prove that \eqref{cdm-eq-lpeps} holds true for small $\eps$ and $R$ well chosen. For that we prove the following 
\begin{lemma}
Assume that $a\in C^{1,\alpha}(\R)$, $J$ satisfying \eqref{hypj1}-\eqref{hypj2} is compactly supported and that 
$\lambda_p(c\df +\m +{\bf a})<0$. Then there exists $R_0$ such that  
$$\limsup_{\eps \to 0,\eps>0}\lambda_p(\eps\ddf +c\df +\mb{R_0} +{\bf a})<0.$$
\end{lemma}
\noindent Assume for the moment that this Lemma holds true. Then there exists $\eps_0(c,R_0)$ such that for all $\eps \le \eps_0$ we have 
$$\lambda_p(\eps\ddf +c\df +\mb{R_0} +{\bf a})<0.$$
By using that $\lambda_p(\eps\ddf +c\df +\mb{R} +{\bf a})$ is monotone non increasing with respect to $R$ we then conclude that for all $R\ge R_0$ and all $\eps \le \eps_0$
we have $$\lambda_p(\eps\ddf +c\df +\mb{R} +{\bf a})\le \lambda_p(\eps\ddf +c\df +\mb{R_0} +{\bf a})<0. $$
 
Now, thanks to the Collatz-Wieland type characterisation of $\lambda_p(\eps\ddf +c\df +\mb{R} +{\bf a})$ (Theorem \ref{cdm-thm-ellip}), to prove this Lemma it is enough to show that there exists $R_0>0$ such that we have $\lambda_p(c\df +\mb{R_0} +{\bf a})<0$ and for small $\delta>0$, say $\delta \le -\frac{\lambda_p(c\df +\mb{R_0} +{\bf a})}{2}$,  we can find $\psi>0, \psi \in W^{2,1}(\O)\cap L^{\infty}(\O)$ and $\eps(\delta)$ such that for all $\eps \le \eps_\delta$  : 

 $$\begin{cases}
 \eps\opddf{\psi}(x) +c\opdf{\psi}(x) +\opmb{\psi}{R_0}(x) +a(x)\psi(x) +(\lambda_p(c\df +\mb{R_0} +{\bf a})+\delta)\psi(x) \ge 0,  \text{ on }    (-R_0,R_0)\\
 \psi(-R_0)=\psi(R_0)=0.
\end{cases}
 $$

To simplify the presentation of the proof of this Lemma we decompose it into two steps. 

\subsubsection*{Step One : Some a priori estimates}
Let $R>0$ and let $\varphi_{p,c,R}$ be the positive eigenfunction associated with $\lambda(c\df+\mb{R}+{\bf a})$. Without any loss of generality, we can  assume that $\varphi_{p,c,R}\le 1$. 
Let $d_1(c,R)$ be  the following positive constant:
$$d_1:=\inf_{x\in (-R,R)}\int_{-R}^{R}J(x-y)\varphi_{p,c,R}(y)\,dy$$

From the equation satisfied by $\varphi_{p,c,R}$  we can check that for all $x\in (-R,R)$ we have
$$
c\opdf{\varphi_{p,c,R}}(x)+(\lambda_p(c\df +\mb{R}+{\bf a})-1+a(x))\varphi_{p,c,R}(x)=-\int_{-R}^{R}J(x-y)\varphi_{p,c,R}(y)\,dy\le -d_1.
$$
Therefore, since for all $c$ and $R$ we have $\ds{\lambda_p(c\df +\mb{R}+{\bf a})\ge -\sup_{x\in (-R,R)}a(x)}$,  we deduce that 

\begin{equation}\label{esti-phiprime}
\opdf{\varphi_{p,c}}(x)\le -\frac{d_1}{c}+\frac{\gamma_0}{c}\varphi_{p,c,R}(x)
\end{equation}
with $\ds{\gamma_0:= \sup_{x\in \R} a(x)-\inf_{x\in \R}a(x)+1}$.

Now we claim that
\begin{claim}\label{cdm-cla-cdelta}
For all $\delta>0$, there exists $c(\delta)>c$, such that for all $c<c'\le c(\delta)$ we have 
\begin{equation}
c\opdf{\varphi_{p,c',R}}(x)+\opmb{\varphi_{p,c',R}}{R}(x)+(a(x)+\lambda_p(c'\df +\mb{R}+{\bf a}) +\delta)\varphi_{p,c',R}(x)\ge  \frac{c'-c}{c'}d_1(c')+ \frac{\delta}{2} \varphi_{p,c',R}(x).
\end{equation} 
\end{claim}
\begin{proof}
Let $\delta>0$ be fixed and  for $c'>c$ and $\varphi_{p,c',}$ let us compute 
$$\ds{c\opdf{\varphi_{p,c',R}}(x)+\opmb{\varphi_{p,c',R}}{R}(x)+(a(x)+\lambda_p(c'\df +\mb{R}+{\bf a}) +\delta)\varphi_{p,c',R}}(x).$$

By  \eqref{esti-phiprime} we can infer that  for  any $c'>c$ 
\begin{align*}
c\opdf{\varphi_{p,c',R}}+\opmb{\varphi_{p,c',R}}{R}+(a(x)+\lambda_p(c'\df +\mb{R}+{\bf a}) +\delta)\varphi_{p,c',R}&=(c-c')\opdf{\varphi_{p,c',R}}+\delta \varphi_{p,c',R},\\
&\ge \frac{c'-c}{c'}d_1(c')+ \left(\delta -\frac{(c'-c)}{c'}\gamma_0\right) \varphi_{p,c',R}.
\end{align*}
Thus, the claim holds true by taking $c(\delta)\le c\left(1+\frac{\delta}{2\gamma_0}\right)$.
\end{proof}
 
\begin{remark}
Note that $c(\delta)$ does not depend of $R$. 
\end{remark}

\subsubsection*{Step two : Construction of a good test function for $\eps$  small}

\begin{proof}[\nopunct]
First, let us remark that since $J$ and $f$ satisfies \eqref{hypj1}-\eqref{hypf2} and $\lambda_p(c\df +\m +{\bf a})<0$ by using  Lemma \ref{cdm-lem-lim}  we have $$\lim_{R\to +\infty}\lambda_{p}(c\df +\mb{R} +{\bf a})=\lambda_p(c\df +\m +{\bf a})<0$$
and so there exists $R_0>0$ such that 
\begin{equation}\label{eq-const1} 
\lambda_{p}(c\df +\mb{R_0} +{\bf a})<\frac{\lambda_{p}(c\df +\m +{\bf a})}{2}<0.
\end{equation}  

\noindent Let us fix  $\delta \le -\frac{\lambda_p(c\df +\mb{R_0} +{\bf a})}{2}$ and let us start our construction of  $\psi$.

First, thanks to the continuity of $\lambda_p(c\df+\mb{R_0}+{\bf a})$ with respect to the speed $c$, there exists $c_0>c$ such that for all $c<c'\le c_0$, we have 
\begin{equation}\label{eq-c1} 
\lambda_p(c'\df+\mb{R_0}+{\bf a})\le \lambda_p(c\df+\mb{R_0}+{\bf a})+\frac{\delta}{4}.
\end{equation}
From our choice of $c_0$ let us observe that for  all $c_0\ge c'>c$ we have 
$$\lambda_p(c'\df+\mb{R_0}+{\bf a})\le \frac{7}{8}\lambda_p(c\df+\mb{R_0}+{\bf a})<\frac{7}{16}\lambda_p(c\df+\m+{\bf a})<0.$$
Now thanks to the Claim \ref{cdm-cla-cdelta} there exists $c(\delta)$ such that for all $c<c'\le c(\delta)$ the principal eigenfunction $\varphi_{p,c',R_0}$ verifies 
$$ 
c\opdf{\varphi_{p,c',R_0}}+\opmb{\varphi_{p,c',R_0}}{R_0}+(a(x)+\lambda_p(c'\df +\mb{R_0}+{\bf a}) +\delta)\varphi_{p,c',R_0}\ge  \frac{c'-c}{c'}d_1(c')+ \frac{\delta}{2} \varphi_{p,c',R_0}.
 $$
Fix now $c'\le \inf\{c_0,c(\delta)\}$ then from the above inequality and by using \eqref{eq-c1} we achieve   
\begin{equation}\label{eq-c2}
c\opdf{\varphi_{p,c',R_0}}+\opmb{\varphi_{p,c',R_0}}{R_0}+(a(x)+\lambda_p(c\df +\mb{R_0}+{\bf a}) +\delta)\varphi_{p,c',R_0}\ge \frac{c'-c}{c'}d_1(c')+ \frac{\delta}{4} \varphi_{p,c',R_0}.
\end{equation}

\noindent For $\tau < \frac{1}{4}$, let us now choose a family of increasing cut-off function $\zeta_\tau(x)$ such that
\begin{equation*}
\begin{cases}
\zeta_\tau(x) =0 \qquad \text{for} \quad -R_0\le x\le -R_0+\tau\\
\zeta_\tau'(x)  \ge 0 \qquad \text{for} \quad -R_0+\tau < x< -R_0 +\sqrt{\tau}\\
\zeta_\tau(x)  =1\qquad \text{for} \quad -R_0+\sqrt{\tau} \le x\le R_0
\end{cases}
\end{equation*}
and that for some uniform constant $C_0$ verifies $|\zeta_\tau'(x)|\le \frac{C_0}{\sqrt{\tau}}$.    
Choose $\tau_0<\frac{1}{4}$ small such that 
$$
\inf_{x\in\O}\int_{-R_0+\sqrt{\tau_0}}^{R_0}J(x-y)\varphi_{p,c',R_0}(y)\,dy :=d_0>0.
$$
Let us define $\psi:=\varphi_{p,c',R_0}(x)\zeta_\tau(x)$ and for $\tau\le \tau_0$ 
let us compute 
$$c\opdf{\psi}+\opmb{\psi}{R_0}+(a(x)+\lambda_p(c\df +\mb{R_0}+{\bf a}) +\delta)\psi. $$
Observe that $\psi\equiv 0$ for  $x\le -R_0+\tau$ and $\tau\le \tau_0$,  so we have for $x\le -R_0+\tau$ 
  
\begin{align}
c\opdf{\psi}+\opmb{\psi}{R_0}+(a(x)+\lambda_p(c\df +\mb{R_0}+{\bf a}) +\delta)\psi&= \int_{-R_0}^{R_0} J(x-y)\psi(y)\,dy \nonumber\\
&\ge  \int_{-R_0+\sqrt{\tau_0}}^{R_0}J(x-y)\varphi_{p,c',R_0}(y)\,dy \ge d_0. \label{eq-c3}
\end{align}
On the other hand, for $x\in (-R_0+\sqrt{\tau}, R_0)$ we have $\zeta_{\tau}(x)=1$ and thanks to \eqref{eq-c2} we deduce that 
\begin{align*}
c\opdf{\psi}+\opmb{\psi}{R_0}+(a(x)+\lambda_p(c\df +\mb{R_0}+{\bf a}) +\delta)\psi&\ge  \int_{-R_0}^{R_0} J(x-y)\varphi_{p,c',R_0}(y)(\zeta_\tau(y)-1)\,dy+ \frac{c'-c}{c'}d_1+ \frac{\delta}{4} \psi. \nonumber\\
&\ge \int_{-R_0}^{-R_0+\sqrt{\tau}}J(x-y)\varphi_{p,c',R_0}(y)\,dy + \frac{c'-c}{c'}d_1+ \frac{\delta}{4} \varphi_{p,c',R_0}. \nonumber\\
&\ge -\|J\|_{\infty}\sqrt{\tau} + \frac{c'-c}{c'}d_1+ \frac{\delta}{4} \varphi_{p,c',R_0}. 
\end{align*}
By taking $\ds{\tau\le \inf\left\{\left(\frac{c'-c}{2c' \|J\|_{\infty}}d_1\right)^2,\tau_0\right\}} $ we then achieve for  $x\in (-R_0+\sqrt{\tau},R_0)$
\begin{equation}\label{eq-c4}
c\opdf{\psi}+\opmb{\psi}{R_0}+(a(x)+\lambda_p(c\df +\mb{R_0}+{\bf a}) +\delta)\psi\ge \frac{c'-c}{2c'}d_1+ \frac{\delta}{4} \varphi_{p,c',R_0}\ge 0.
\end{equation}
Similarly,  thanks to \eqref{eq-c2} and since $\zeta_\tau$ is a monotone non decreasing function,  a short computation shows that on $(-R_0+\tau,-R_0+\sqrt{\tau})$
\begin{multline*}
c\opdf{\psi}+\opmb{\psi}{R_0}+(a(x)+\lambda_p(c\df +\mb{R_0}+a) +\delta)\psi\ge \int_{-R_0}^{R_0} J(x-y)\varphi_{p,c',R_0}(y)(\zeta_\tau(y)-\zeta_{\tau}(x))\,dy\\+\zeta_\tau(x)\left(\frac{c'-c}{c'}d_1+ \frac{\delta}{4} \varphi_{p,c',R_0}\right).
\end{multline*}
Define $\ds{I:=\int_{-R_0}^{R_0} J(x-y)\varphi_{p,c',R_0}(y)(\zeta_\tau(y)-\zeta_{\tau}(x))\,dy,}$
and let us estimate $I$.\\ 
Since $\zeta_\tau\equiv 1$ in $(-R_0+\sqrt{\tau},R_0)$, $\zeta_\tau$ is monotone non decreasing and $|\zeta_\tau'(x)|\le \frac{C_0}{\sqrt{\tau}}$  and since $\tau\le \tau_0$ and $x\in (-R_0+\tau,-R_0+\sqrt{\tau})$   we can estimate  $I$ as follows:
\begin{align*}
I&\ge (1-\zeta_\tau(x))\int_{-R_0+\sqrt{\tau}}^{R_0}J(x-y)\varphi_{p,c'}(y)\,dy +\int_{-R_0}^{x}J(x-y)\varphi_{p,c'}(y)(\zeta_\tau(y)-\zeta_\tau(x))\,dy,\\
&\ge (1-\zeta_\tau(x))d_0 -\int_{-R_0}^{x}J(x-y)\left|\frac{\zeta_\tau(y)-\zeta_\tau(x)}{y-x}\right||x-y|\,dy,\\
&\ge (1-\zeta_\tau(x))d_0 -\frac{C_0}{\sqrt{\tau}}\int^{\sqrt{\tau}}_{0}J(z)|z|\,dz,\\
&\ge (1-\zeta_\tau(x))d_0 -\frac{C_0\sqrt{\tau}\|J\|_{\infty}}{2}.
\end{align*}
As a consequence, by setting $d^*:=\inf\left\{d_0,\frac{c'-c}{2c'}d_1 \right\}$  and by choosing $\tau$ small, say $$\ds{\tau\le \tau_1:= \inf\left\{\tau_0, \left(\frac{c'-c}{2c' \|J\|_{\infty}}d_1\right)^2, \left(\frac{d^*}{C_0\|J\|_{\infty}}\right)^2\right\}}, $$ we then achieve for all $x\in (-R_0+\tau,-R_0+\sqrt{\tau})$
\begin{equation}\label{eq-c5}
c\opdf{\psi}+\opmb{\psi}{R_0}+(a(x)+\lambda_p(c\df +\mb{R_0}+{\bf a}) +\delta)\psi\ge \zeta_\tau(x)\frac{\delta}{4} \varphi_{p,c',R_0}(x)+ \frac{d^*}{2}.
\end{equation}

By using that $a\in C^{1,\alpha}(\R)$ we can check that the eigenfunction $\varphi_{p,c',R_0}\in C^{2}((-R_0,R_0))$ and by definition of $\psi$  this implies that $\psi\in C^{2}((-R_0,R_0))$.
Moreover collecting \eqref{eq-c3}, \eqref{eq-c4} and \eqref{eq-c5}, for any $x\in (-R_0,R_0)$, we also have 
\begin{align*}
&\eps \opddf{\psi} + c\opdf{\psi}+\opmb{\psi}{R_0}+(a(x)+\lambda_p(c\df +\mb{R_0}+{\bf a}) +\delta)\psi\ge \eps \opddf{\psi} + \frac{d^*}{2}\\
&\psi(-R_0)=\psi(R_0)=0
\end{align*}
which for $\eps$ small enough says $\eps \le \eps^*:= \frac{d^*}{4\|\psi\|_{C^2}}$ satisfies 

 $$ \eps \opddf{\psi} + c\opdf{\psi}+\opmb{\psi}{R_0}+(a(x)+\lambda_p(c\df +\mb{R_0}+{\bf a}) +\delta)\psi\ge \frac{d^*}{4}>0,$$
ending thus the proof of the Lemma.
\end{proof}

\begin{remark}\label{cdm-rem-psi}
From the above construction, since $\psi\equiv 0$ in $(-R_0,-R_0 +\tau)$, we can easily  observe that  for all $\eps \le \eps^*$ the function $\psi$ is still a good test function in $(r_1,R_0)$ for all $r_1\le -R_0$. 
\end{remark}

\begin{remark}\label{cdm-rem-kappapsi}
Observe that for small $\kappa$, the function $\kappa\psi$ can also serve as a sub-solution of the problem \eqref{cdm-eq-reg-kpp}-\eqref{cdm-eq-reg-kpp-bc} defined in $(-R_0,R_0)$. Indeed, a straightforward computation shows that 
\begin{align*}
\eps \opddf{\kappa\psi} + c\df\kappa\psi+\opmb{\kappa \psi}{R_0} +f(x,\kappa \psi)&\ge \frac{\kappa d^*}{4}+(-\lambda_p(c\df+\mb{R_0}+{\bf a} )-\delta)\kappa \psi + o(\kappa \psi)\\
&\ge  \frac{\kappa d^*}{4}-\frac{\lambda_p(c\df+\mb{R_0}+{\bf a} )}{2}\kappa \psi +o(\kappa\psi)\ge 0.
\end{align*}
Thus since  large constants are super-solutions of the problem and the problem admits a unique positive solution we will have 
$$\kappa \psi \le u_{\eps,R_0} $$ 
\end{remark}

\subsection{Existence of a non-trivial positive solution in $\R$}\label{cdm-ss-constr2}
Thanks to the previous subsection,  for a fixed $c>0$ we know that there exists $\eps_0>0$ and $R_0>0$ such that  for any $0<\eps\le \eps_0$ and $R\ge R_0$ there exists $u_{\eps,R}$ a unique positive solution to \eqref{cdm-eq-reg-kpp}-- \eqref{cdm-eq-reg-kpp-bc}. 
 In addition, for a fixed $\eps$  and for any $R_1>R_2\ge R_0$, the solution $u_{\eps, R_1}$ is a super-solution for the problem  
\begin{align*}
 &\eps\opddf{u}+c\opdf{u} +\opmb{u}{R_2}+f(x,u)=0 \quad \text{in}\quad  (-R_2,R_2)\\
&u(-R_2)=u(R_2)=0 
 \end{align*}
Therefore by a sweeping argument we get
$$  u_{\eps,R_2}\le u_{\eps,R_1} \quad \text{in}\quad (-R_2,R_2).$$
Thus, for $\eps$ fixed the map $R\mapsto u_{\eps,R}$ is monotone increasing. 
As a consequence and thanks to Remark \ref{cdm-rem-kappapsi} for all $R\ge R_0$ we will always have  
\begin{equation}\label{eq-norma}
\kappa \psi \le u_{\eps,R_0}\le u_{\eps,R} \qquad \text{ on }\qquad (-R_0,R_0) 
\end{equation}
for some $\kappa>0$ independent of $\eps$.

On the other hand, since $f$ satisfies \eqref{hypf1}-\eqref{hypf2}, we can find a constant $M$ such that for all $R\ge R_0$ this constant is a super-solution of the problem \eqref{cdm-eq-reg-kpp}-- \eqref{cdm-eq-reg-kpp-bc}. Thus, by uniqueness of the solution, we then have $u_{\eps,R}\le M$ for all $\eps \le \eps_0$ and $R\ge R_0$.

Let us fix now $\eps\le \eps_0$  and let  $(R_n)_{n\in \N}$ be an increasing sequence starting from $R_0$ and such that $R_n\to +\infty$.  Let us denote by $(u_{n})_{n\in \N}$  the corresponding sequence of solutions of \eqref{cdm-eq-reg-kpp}--\eqref{cdm-eq-reg-kpp-bc} set on $(-R_n,R_n)$. Since $u_n$ is uniformly bounded, by using local elliptic estimate we can check that $(u_n)_{n \in \N}$ is bounded uniformly in $C^{2,\alpha}(-R_n,R_n)$ and therefore by a diagonal extraction process we can extract of the sequence $(u_n)_{n\in \N}$ a subsequence, still denoted $(u_n)_{n\in \N}$ that converges locally uniformly to a positive bounded function $u_\eps$ which is solution to 
\begin{equation*}
 \eps\opddf{u_\eps}+c\opdf{u_\eps} +\opm{u_\eps}+f(x,u_\eps)=0 \quad \text{in}\quad  \R.
 \end{equation*}
  Moreover, thanks to \eqref{eq-norma}, we have $\kappa \psi \le  u_\eps\le M$ in $(-R_0,R_0)$. Note that the latter bound is independent of $\eps$ so we get 
 for all $\eps\le \eps_0$ 
  \begin{equation}\label{eq-norma-eps}
  \kappa \psi\mathds{1}_{(-R_0,R_0)}  \le u_\eps \le M.
  \end{equation}

  Whence, for all $\eps \le \eps_0$ there exists a non trivial solution $u_\eps$ to \eqref{cdm-eq-eps} thus proving Theorem   \ref{cdm-thm-eps}.
  
\subsection{Uniqueness of the solution of \eqref{cdm-eq-eps}} \label{cdm-ss-constr3}
Let us now prove that the solution $u_\eps$ is unique. We argue by contradiction and assume that there exists $v$ another positive solution to \eqref{cdm-eq-eps}.
We will show that $u_\eps<v$ and $v\le u_\eps$ which will give us our contradiction.  

Let us first remark that  $v$ is a supersolution of problem \eqref{cdm-eq-reg-kpp} -- \eqref{cdm-eq-reg-kpp-bc} for any $R>0$. Therefore by a standard sweeping argument we conclude that $v\ge u_{\eps,R}$ for all $R\ge R_0$. Since $u_{\eps,R}$ is monotone with respect to $R$, it follows that $\ds{v\ge u_\eps}$.   
By assumption  $v\not\equiv u_\eps$ almost everywhere and by using the strong maximum principle we then have $v>u_\eps$.  Our main task is then to prove the inequality $v\le u_\eps$. To do so we will use a sweeping type argument. 
But first we will establish some useful properties of $v$ that will constantly used along our proof. Namely,

\begin{proposition}\label{cdm-prop-v-l1}
Assume that $f$ and $J$ satisfy $\eqref{hypj1}-\eqref{hypf2}$ and assume further that $J$ satisfies \eqref{hypj3}.Let $v\in C^2_{loc}(\R)$ be a positive bounded solution to \eqref{cdm-eq-eps}. Then
$v\in L^1(\R)$. 
\end{proposition}
 
\begin{proof}
By assumption since $f$ is in $C^1$ and  $v\in L^{\infty}(\R)\cap C^2_{loc}(\R)$ is a solution to \eqref{cdm-eq-eps}, by using  interior elliptic regularity we deduce that   $\opdf{v} \in L^{\infty}(\R)$. 
Now  to prove that $v\in L^1$ and let us  integrate \eqref{cdm-eq-eps}  over $(-R,R)$, we then get 
$$\eps (\opdf{v}(R)-\opdf{v}(-R))+c(v(R)-v(-R))+\int_{-R}^{R}\int_{\R}J(-z)[v(x+z)-v(x)]\,dz\,dx =-\int_{-R}^{R}f(x,v(x))\,dx.$$
Let us set $\a_{R}:=\eps (\opdf{v}(R)-\opdf{v}(-R))+c(v(R)-v(-R))$ and by using that $J$ has a first moment, that $v$ is smooth, and  Fubini's Theorem we can  rewrite the equality as follows
\begin{align*}
-\int_{-R}^{R}f(x,v(x))\,dx &=\a_{R}+\int_{-R}^{R}\int_{\R}J(-z)z\int_{0}^1\opdf{v}(x+sz)\,ds dz dx \\
&= \a_{R}+ \int_{0}^{1}\int_{\R}J(-z)z[v(R+sz)-v(-R+sz])\,ds dz
\end{align*}
Therefore, since $v$ and $\opdf{v}$ are bounded we get for all $R>0$      
\begin{equation}\label{cdm-eq-uni-1}
-\int_{-R}^{R}f(x,v(x))\,dx< C_0:=2\|v\|_{\infty}\left(c+\int_{\R}J(z)|z|\,dz\right)+2\eps \|\opdf{v}\|_{\infty}.
\end{equation}
Now since $f$ satisfies \eqref{hypf1} and \eqref{hypf2} we can find  $R_0>0$ and $\kappa_0>0$ such that $f(x,s)<-\kappa_0 s$ for all $|x|>R_0$ and $s \in \R^+$.  Whence,  for all $R>R_0$ we get from \eqref{cdm-eq-uni-1} that  

\begin{align*}
-\int_{-R}^{-R_0}f(x,v(x))\,dx -\int_{R_0}^{R}f(x,v(x))\,dx&\le \left|\int_{-R_0}^{R_0}f(x,v(x))\,dx\right|  + C_0\\
 \kappa_0\left(\int_{-R}^{-R_0}v(x)\,dx+\int_{R_0}^{R}v(x)\,dx\right)&\le C_0+ 2R_0\|f\|_{\infty},
\end{align*}
which by using that $v$ is bounded implies that 

\begin{equation}\label{cdm-eq-uni-2}
\int_{-R}^{R} v(x)\,dx \le \frac{C_0+ 2R_0[\|f\|_{\infty}+\|v\|_{\infty}]}{\kappa_0}.
\end{equation}
Since $v$ is positive, the later estimate shows that $v\in L^{1}$.

\end{proof} 

\begin{remark}\label{cdm-rem-lp}
Note that as a corollary of the fact $v\in L^1$, since  $v$  is smooth we must have $\ds{\lim_{x\to\pm\infty} v(x)=0}$. In addition since $v\in L^{1}(\R)\cap L^{\infty}(\R)$, by interpolation $v \in L^p$ for any $1\le p$. In particular, $v\in L^2(\R)$.
\end{remark}
\begin{remark} \label{cdm-rem-l1}
We observe that the above proof only rely on elementary computations which will be true as well when $\eps=0$.  Therefore we then have that any bounded positive smooth solution to \eqref{cdm-eq} is in $L^1(\R)$.
\end{remark}

Equipped with the Proposition \ref{cdm-prop-v-l1}, we are now in position to prove $v\le u_\eps$. To this end, we will use a sweeping type argument,which essentially relies on two main steps. First we prove that $v\le \tau_0 u_\eps$ for some $\tau_0>0$ and then in a second step we show that $v\le \tau u_\eps$ for all $\tau \ge 1$. Let us start by  proving the following:
\begin{claim}\label{cdm-prop-sweep1}
There exists $\tau_0>0$ such that 
$\ds{v(x)\le \tau_0 u_\eps(x)\,  \text{ for all } x\in \R}$ 
\end{claim}

\begin{proof}
Let us  define $b_1(x):=\frac{f(x,u_\eps(x))}{u_\eps(x)}$ and $b_2(x):=\frac{f(x,v(x))}{v(x)}$. Since $f$ satisfies \eqref{hypf1}--\eqref{hypf3} and $v>u_\eps$ we have $b_1(x)>b_2(x)$ and therefore $u_\eps,v$ satisfy:
\begin{align*}
&\eps \opddf{u_\eps}(x) +c\opdf{u_\eps}(x)+\opm{u_\eps}(x)+b_2(x)u_\eps(x)\,<\,0\\
&\eps \opddf{v}(x) +c\opdf{v}(x)+\opm{v}(x)+b_2(x)v(x) \,=\,0. 
\end{align*}

By using that $f$ satisfies \eqref{hypf2} and since $v$ is bounded we can find $R_0$ and $\nu_0>0$ such that $b_2(x)<-\nu_0$ for all $|x|\ge R_0$. Now since $u_\eps$  
and $v$ are bounded and positive the following positive quantity are well defined
$$M_0:=\max_{x\in [-R_0,R_0]}v(x) \qquad m_0:=\min_{x\in[R_0,R_0]}u_\eps(x).$$
By considering now the function $C u_\eps$ with $C:=\frac{M_0}{m_0}+1$ we then have $C u_\eps>v$ in $[-R_0,R_0]$. Therefore since $0=\lim_{|x|\to +\infty}v(x)\le \lim_{|x|\to +\infty}Cu_\eps (x) $, by using the weak comparison principle, Theorem \ref{cdm-thm-wcp}, we then conclude that $v\le Cu_\eps$.
\end{proof}

From the Claim \ref{cdm-prop-sweep1}  the following quantity is then well defined  
$$\tau^*:=\inf\{\tau >0, \,|\, v\le \tau u_\eps \},$$
 and we claim that 
 
\begin{claim} 
 $\tau^*\le 1$
 \end{claim}
 
 Note that by proving the claim we end our proof since we then achieved that $v\le \tau^*u\le u$.
\begin{proof}
Again, to show this claim we will argue by contradiction and assume that $\tau^*>1$. By definition we have $v \le \tau^* u_\eps$ and by the strong comparison principle, Theorem \ref{cdm-thm-scp}, either $v< \tau^* u_\eps$ or $v\equiv \tau^* u_\eps$.   In the later case,  we get the contradiction 
 \begin{align*}
 0=\eps \opddf{v}+ c \opdf{v}+ \opm{v} +\frac{f(x,v)}{v}v  &= \eps \opddf{\sigma^*u_\eps}+ c \opdf{\sigma^*u_\eps}+ \opm{\sigma^* u_\eps} +\frac{f(x,\sigma^*u_\eps)}{\sigma^* u_\eps}\sigma^*u_\eps\\
                   &<\sigma^*\left( \eps \opddf{u_\eps}+ c \opdf{u_\eps}+ \opm{u_\eps} +f(x,u_{\eps})\right)=0.
 \end{align*}

Let us now obtain a contradiction in the other situation. Since $v<\tau^* u_\eps$, we may find $\tau'<\tau^*$ such that $v(x)< \tau' u_\eps(x)$ for all $x\in [-R_0,R_0]$.  In addition  we can check that  $v$ and $\tau'u_\eps$ satisfy 
 $$\begin{cases}
\eps \opddf{\tau'u_\eps}(x) +c\opdf{\tau'u_\eps}(x)+\opm{\tau'u_\eps}(x)+b_2(x)(\tau'u_\eps(x)) \,<\,0 & \text{ for all }  |x|>R_0 \\
\eps \opddf{v}(x) +c\opdf{v}(x)+\opm{v}(x)+b_2(x)v(x) \,=\,0 & \text{ for all }  |x|>R_0,\\
v(x)<\tau'u_\eps(x) &\text{ for all }  x\in [-R_0,R_0],\\
\lim_{|x|\to \infty} v(x)\le \lim_{|x|\to +\infty}\tau'u_\eps (x).
\end{cases}
$$  
Thus by applying the weak comparison principle, Theorem \ref{cdm-thm-wcp}, we get $v\le \tau' u_\eps$ contradicting the definition of $\tau^*$.
Whence we must have $ \tau^* \le 1$.
\end{proof}

 \section{The sufficient condition: Existence of a non trivial steady state}\label{cdm-section-crit}

In this section, by analysing the singular limit problem \eqref{cdm-eq-eps} as $\eps\to 0$ we construct a positive  solution $\bar u$ of \eqref{cdm-eq} and prove the  sufficient condition stated in Theorem \ref{cdm-thm1}. The proof being similar for $c>0$ and $c<0$, we only present the construction for the case $c>0$ and when necessary we add some remarks  to handle the case $c<0$.


So let us now analyse the singular limit of $u_\eps$ when $\eps \to 0$ and construct a positive non trivial solution to \eqref{cdm-eq}. From \eqref{eq-norma-eps} we know that $ M\ge u_\eps\ge \kappa \psi\mathds{1}_{(-R_0,R_0)} $ with $M$ and $\kappa $ constants independent of $\eps $. Next, we show that there exists $C_1$ such that for all $\eps$, we have 
 \begin{equation}\label{eq-esti-C1}
 |\opdf{u_\eps}|\le C_1
 \end{equation}

To obtain such estimate, we will first establish the following  estimate:

\begin{proposition}
Assume $J$ and $f$ satisfies \eqref{hypj1}-\eqref{hypf2}, then for all   $0<\eps<\eps_0$ there exists $C_\eps$ such that 
$$\|u_\eps\|_{H_1(\R)}\le C_\eps.$$ 
 \end{proposition}
 
Assume for the moment that the proposition holds true and let us prove that it implies \eqref{eq-esti-C1}. Indeed, since $u_\eps$ is bounded in $H_1(\R)$ and smooth we can then infer that  $\lim_{|x|\to +\infty}\opdf{u_\eps}(x)=0$ and $\lim_{|x|\to +\infty}u_\eps(x)=0$  and  as a consequence, the quantities $\ds{\sup_{x\in \R}\opdf{u_\eps}(x)}$ and $\ds{\inf_{x\in \R}\opdf{u_\eps}(x)}$ are achieved at some points $x_0, x_1$ in $\R$. Now,  by using \eqref{cdm-eq-eps} we deduce that at these points we have 
$$-M-\sup_{x\in \R,s \in [0,M]}|f|(x,s) \le c\opdf{u_\eps}(x_i)\le M+\sup_{x\in \R,s \in [0,M]}|f|(x,s)$$
and thus $$|\opdf{u_\eps}(x_i) |\le C_1:=\frac{M+\sup_{x\in \R,s \in [0,M]}|f|(x,s)}{c}. $$

Now thanks to  \eqref{eq-esti-C1} we can construct a non trivial solution to \eqref{cdm-eq}. Indeed,  take a sequence $(\eps_n)_{n\in\N}$ such that $\eps_n \to 0$ and let $u_{\eps_n}$ be the  non trivial solution of the regularised equation \eqref{cdm-eq-eps} with $\eps_n$.  Thanks to \eqref{eq-esti-C1} and \eqref{eq-norma-eps}  the sequence  $(u_{\eps_n})_{n\in \N}$ is uniformly bounded in $C^{1}(\R)$  and therefore by a diagonal extraction process we can find a subsequence that converges locally uniformly and pointwise to a non negative   function $\tilde u$. Moreover,  $\kappa \psi\mathds{1}_{(-R_0,R_0)} \le \tilde u \le M$ and passing to the limit in the equation in the sense of distribution, we can check that $\tilde u$ satisfies in a weak sense
$$ c\opdf{\tilde u} +\opm{\tilde u}+f(x,\tilde u)=0.$$ 
Recall that $c\neq 0$, thereby by a standard regularity argument,  we can see that  $\tilde u$ is smooth and satisfies the above equation strongly whence
$$ c\opdf{\tilde u}(x) +\opm{\tilde u}(x)+f(x,\tilde u(x))=0 \quad \text{for all }\quad x \in  \R.$$

To finish our proof let us prove the proposition.
 
 \begin{proof}[Proof of the Proposition]
To prove that, let us first recall that since $\eps>0$ and $u_\eps$ is a bounded solution of \eqref{cdm-eq-eps} by using standard regularity estimates it follows  that  $\opdf{u_\eps}$ is also uniformly bounded in $\R$. 
So by multiplying \eqref{cdm-eq-eps} by $u_\eps$ and by integrating the resulting equation over $(-R,R)$ we  then obtain 
$$\eps\int_{-R}^Ru_\eps\opddf{u_\eps}(x)\, dx  + \frac{c}{2}\int_{-R}^R\opdf{u_\eps^2}(x)\,dx + \int_{-R}^R(u_\eps(x)\opm{u_\eps}(x) -u_\eps^2(x))\,dx + \int_{-R}^Rf(x,u_\eps(x))u_\eps(x)\,dx=0.$$ 
By  integrating by part the first terms and rearranging the other terms it then follows that 
 
\begin{equation}\label{eq-h1}
\eps \int_{-R}^R |\opdf{u_\eps}|^2\,dx- \int_{-R}^R(u_\eps(x)\opm{u_\eps}(x) -u_\eps^2(x))\,dx=\int_{-R}^Rf(x,u_\eps(x))u_\eps(x)\,dx + I_{R}
\end{equation} 
where $I_R$ is the quantity
$$I_R:=
\frac{c}{2}[u_\eps^2(R)-u_\eps^2(-R)] + u_\eps(R)\opdf{u_{\eps}}(R) -u_\eps(-R)\opdf{u_{\eps}}(-R). 
$$
Let us estimates all the terms in the equality \eqref{eq-h1}. By  Proposition \ref{cdm-prop-v-l1} and Remark \ref{cdm-rem-lp}  it follows that $u_\eps \in L^2$  and  $u_{\eps}(\pm R)\to 0$ as $R\to +\infty$, and thus 
$\ds{
\lim_{R \to \infty} I_R =0} $
since
 $|\opdf{u_{\eps}}|$ is bounded uniformly.
 As a consequence there exists a positive constant $C_1$ such that for all $R>0$
\begin{equation}\label{eq-h1-1}
I_R\le C_1.
\end{equation} 
 
Now by recalling that $f$ satisfies  \ref{hypf1}-\ref{hypf2} and  $\ds{0<u_\eps\le M:=\|S\|_{\infty}}$  since by Proposition \ref{cdm-prop-v-l1} $u_\eps$ is integrable  we also deduce that for some positive constant $C_2$
\begin{equation}\label{eq-h1-2}
\left|\int_{-R}^Rf(x,u_\eps)u_\eps(x)\,dx\right|\le \int_{-R}^R|f(x,u_\eps(x))|u_\eps(x)\,dx \le \sup_{x\in \R,s \in [0,M]}|f(x,s)|\int_{-R}^{R}u_\eps(x)dx\le C_2.
\end{equation}

Lastly, since $u_\eps$ is bounded and integrable we also have the following estimate for some positive constant $C_3$: 

\begin{align}
-\int_{-R}^R(u_\eps(x)\opm{u_\eps}(x) -u_\eps^2(x))\,dx &\ge\int_{-R}^{R}u_\eps^2(x)\,dx - \int_{-R}^{R}\int_{\R}J(x-y)u_\eps(x) u_\eps(y)\,dxdy\nonumber \\
&\ge \int_{-R}^{R}u_\eps^2(x)\,dx - M\|J\|_{\infty}\int_{-R}^{R}u_\eps(x)\,dx\nonumber\\
 &\ge \int_{-R}^{R}u_\eps^2(x)\,dx - C_3.\label{eq-h1-3}
\end{align}

Combining \eqref{eq-h1} with \eqref{eq-h1-1},\eqref{eq-h1-2} and \eqref{eq-h1-3}, we then achieve for all $R>0$

$$ \eps \int_{-R}^R |\opdf{u_\eps}|^2\,dx+ \int_{-R}^R u_\eps^2\,dx\le C_1+C_2+C_3<+\infty.$$
 \end{proof}

\section{Uniqueness}\label{cdm-section-uniq}

Having constructed a  smooth positive solution of \eqref{cdm-eq} in $L^1(\R)$, we continue our proof of Theorem \ref{cdm-thm1} by proving  its  uniqueness. Unfortunately, the argument used to proved the uniqueness of $u_\eps$ does not applies here and we require a new approach.  
To obtain the uniqueness of $u$, we argue by contradiction and assume  that  $v\in C^1(\R)\cap L^{\infty}(\R)$ is another positive solution. Our argumentation is rather long and for convenience we decompose it three subsections    

\subsection{$v$ is $H^1$}
Let us show that the solution $v$ is in $H^1$. Thanks to Proposition \ref{cdm-prop-v-l1} and Remarks \ref{cdm-rem-l1} and \ref{cdm-rem-lp}, we already know that $v \in L^2$ and we only need to prove that $\opdf{v} \in L^2(\R)$.  

To do so, let us  multiply the equation satisfied  by $v$ by $\opdf{v}$ and integrate it over $(-R,R)$.  By rearranging the term, we get 

\begin{align*}
c\int_{-R}^R(\opdf{v})^2&= \int_{-R}^Rv\opdf{v} -\int_{-R}^{R}\int_{\R}J(x-y)\opdf{v}(x)v(y)\,dxdy - \int_{-R}^{R}\frac{f(x,v)}{v}v\opdf{v}\\
&\le \frac{1}{2}[v^{2}(R)-v^{2}(-R)] +\|\opdf{v}\|_{\infty}\left(\int_{-R}^{R}\int_{\R}J(x-y)v(y)\,dxdy + \|f_s(x,0)\|_{\infty}\|v\|_{L^1}\right)\\
&\le \|v\|_{\infty} +\|\opdf{v}\|_{\infty}\|v\|_{L^1}\left(1 + \|f_s(x,0)\|_{\infty}\right).
\end{align*}

\subsection{On the sign of some  principal eigenvalue }
 Let us assume there exists $v\in L^{1}(\R)\cap C^{0,1}(\R)$ such that  $v>0$ and $v$ satisfies
 \begin{equation}\label{cdm-eq-sign1}
 c\opdf{v}(x) +\opm{v}(x) +f(x,v(x))\ge 0 \quad \text{for almost every} \quad x\in \R.
\end{equation}  
 Set $b(x):=\frac{f(x,v(x))}{v(x)}$ and let us  consider the principal eigenvalue $\lambda_p(c\df +\m +{\bf b})$ of the operator $c\df +\m+{\bf b}$ where $\lambda_p$ is defined in sub-section \ref{cdm-ss-pge}. Let us also define the following quantity
$$\lambda_p^{\prime\prime}(-c\df +\ms +{\bf b}):=\inf\{\lambda\in \R\,|\, \exists\, \varphi>0, \varphi\in W^{1,\infty}(\R)\cap C^{1}(\R) \quad\text{s.t.}\quad c\opdf{\varphi}+\opm{\varphi}+(b(x)+\lambda)\varphi\ge 0 \}. $$

We claim that 
\begin{proposition}\label{cdm-prop:lp}
Assume that $f$ satisfies \eqref{hypf1}-\eqref{hypf2} and $J$ satisfies \eqref{hypj1}-\eqref{hypj2}  then 
$$ \lambda_p^{\prime\prime}(-c\df+\ms+{\bf b}) \le \lambda_p(c\df+\m+{\bf b})\le 0.$$ 
Moreover, there exists $\varphi>0$, $\varphi\in W^{1,1}(\R)\cap W^{1,\infty}(\R)\cap C^1(\R)$ such that  
$$-c\opdf{\varphi}+\opms{\varphi} + b(x)\varphi \ge 0.$$
\end{proposition} 
 
 \begin{proof}
 We split our proof into two main steps, namely we start by showing that  
 \begin{equation}
 \label{cdm-eq-goal1}
 \lambda_p(c\df +\m+{\bf b})\le 0,
 \end{equation} then we prove that 
 \begin{equation}
 \label{cdm-eq-goal2}
  \lambda_p^{\prime\prime}(-c\df +\ms+{\bf b})\le \lambda_p(c\df +\m+{\bf b}).
 \end{equation}
 The existence of $\varphi$ will come as a side results of the proof of \eqref{cdm-eq-goal2}.

\subsubsection*{Step One : $\lambda_p(c\df +\m+{\bf b})\le 0$ }
To prove \eqref{cdm-eq-goal1}, we argue by contradiction and assume  that $\lambda_p(c\df +\m+{\bf b})>0$. 
\\ Let $\ds{0<\rho<\lambda_p(c\df +\m+{\bf b})}$ to be fixed later on. Then by definition of $v$ we have
$$c\opdf{v}(x)+\opm{v}(x) +b(x)v(x)+\rho v(x)\ge \rho v(x)>0 \quad \text{for almost every} \quad x\in \R.$$
Let $\iota>0$ be a small parameter to be fixed later on, and let us define $\ds{b_\iota(x):=\frac{f(x,v(x)+\iota)}{v(x)+\iota}}$. 
Since $f$ satisfies \eqref{hypf2} and \eqref{hypf3}, we have $b_\iota (x)< b(x)$ and there exists $C_0>0$ such that  $$\|b_\iota -b\|_{\infty}\le C_0\iota.$$
Since $b_\iota(x)\ge b(x)-\|b_\iota-b\|_{\infty}\ge  b(x)-C_0\iota$ by taking $\iota:=\frac{\rho}{2C_0}$, then we easily check that $v$ then satisfies

$$c\opdf{v}(x)+\opm{v}(x) +b_\iota(x)v(x)+\rho v(x)\ge (\rho -C_0\iota)v\ge \frac{\rho}{2}v>0 \quad \text{for almost every} \quad x\in \R.$$

Pick now $\zeta_{\tau}\in C^{\infty}_c(\R)$ a positive symmetric mollifier whose support is include in $[-\tau,\tau]$ and define $v_\tau:=\zeta_\tau \star v(x)$.
Then using the definition of $c\df +\m$ and \eqref{cdm-eq-sign1} we can check that $v_\tau$ satisfies
$$c\opdf{v_\tau}(x)+\opm{v_\tau}(x)+b_\iota(x)v_\tau(x)+\rho v_\tau(x) + \int_{\R}\zeta_\tau(x-y)\left[b_\iota(y) - b_\iota(x)\right]v(y)\,dy \ge  \frac{\rho}{2} v_\tau(x). $$

\noindent Recall that $v\in C^{0,1}(\R)$ and $f(\cdot,s)\in C^{0,1}(\R)$, since by definition $\inf_{\R}(v(x)+\iota)>0$ we have   $b_\iota \in C^{0,1}(\R)$ and therefore there exists $\kappa>0$ such that for all $x,y\in \R$
$$\left|b_\iota(y) - b_\iota(x)\right|\le \kappa |x-y|.$$
As a consequence, since $supp(\zeta_\tau) \subset [-\tau,\tau] $ we have 
 \begin{align*}
c\opdf{v_\tau}(x)+\opm{v_\tau}(x)+b_\iota(x)v_\tau(x)+\rho v_\tau(x) &\ge \frac{\rho}{2} v_\tau(x) -\left|\int_{\R}\zeta_\tau(x-y)\left[b_\iota(y) - b_\iota(x)\right]v(y)\,dy\right|  \\
&\ge \frac{\rho}{2} v_\tau(x) -\kappa\int_{\R}\zeta_\tau(x-y)|x-y|v(y)\,dy  \\
&\ge \left(\frac{\rho}{2} -2\kappa \tau\right)v_\tau(x).  
\end{align*}
 
 By choosing $\tau$ small enough, say $\tau \le\tau_0:=\frac{\rho}{8\kappa}$, we then achieve 
that  
\begin{equation}\label{cdm-eq-sign2}
c\opdf{v_\tau}(x)+\opm{v_\tau}(x)+b(x)v_\tau(x)+\rho v_\tau(x) >0 \quad \text{ for all } \quad x\in \R.
\end{equation}

On the other hand since $b_\iota<b$, by Proposition \ref{cdm-prop-pev} we have $ \lambda_p(c\df +\m+{\bf b})\le \lambda_p(c\df +\m+{\bf b_\iota})$ and therefore $\rho <\lambda_p(c\df +\m+{\bf b_\iota})$. So by definition of $\lambda_p$,   there exists $\psi>0,\psi \in C^1(\R)$ such that 
\begin{equation}\label{cdm-eq-sign3}
c\opdf{\psi}(x)+\opm{\psi}(x)+b_\iota(x)\psi(x)+\rho \psi(x)<0 \quad \text{ for all } \quad x\in \R.
\end{equation}

Now let us recall that $f$ satisfies \eqref{hypf2}, so  there exists  positive constant $R_0$ and $\nu>0$ such that $\frac{f(x,s)}{s}<-\nu$ for all $|x|\ge R_0$, $s\ge 0$. Let us now fixed $\rho:=\frac{1}{2}\min\{\nu,\lambda_p(c\df +\m+{\bf b_\iota})\}$ and
on $[-R_0,R_0]$ let us define the following constant
 $$ M:=\sup_{x\in (-R_0,R_0)}v_\tau(x)\quad\qquad  m:= \inf_{x \in (-R_0,R_0)}\psi(x).$$ 
 Next we  consider the function $\Psi:=C\psi$ with $
\ds{C:=\frac{M}{m}+1.}$ 
Since $v\in L^1(\R)\cap C^{0,1}(\R)$ $v(x) \to 0$ as $|x|\to +\infty$ and so does $v_\tau$ therefore  by using \eqref{cdm-eq-sign2} and \eqref{cdm-eq-sign3} and by  definition of $C$,  we then have  
\begin{equation}\label{cdm-eq-uni-g1}
\begin{cases}
c\opdf{v_\tau}(x)+\opm{v_\tau}(x)+b_\iota(x)v_\tau(x)+\rho v_\tau(x)> 0 &\quad \text{ for all } \quad x \in \R\\
c\opdf{\Psi}(x)+\opm{\Psi}(x)+b_\iota(x)\Psi(x)+\rho \Psi(x)\le 0 &\quad \text{ for all } \quad x \in \R\\
v_\tau(x)<\Psi(x)  &\quad \text{ for all }\quad x \in  [-R_0,R_0],\\
0=\lim_{|x|\to \infty} v_\tau(x)\le \lim_{|x|\to \infty} \Psi(x).
\end{cases}
\end{equation}

By successively applying the weak and the strong comparison principle, Theorems \ref{cdm-thm-wcp} and  \ref{cdm-thm-scp}, we then get $v_\tau<\Psi$ and 
we can define
$\gamma^*:=\inf \{\gamma\ge 0 | v_\tau\le \gamma \Psi\}.$
Since for all $\gamma\ge 0$ the function $\gamma \Psi$ satisfies, 
$$ c\opdf{\gamma\Psi}(x)+\opm{\gamma\Psi}(x)+b_\iota(x)\gamma\Psi(x)+\rho \gamma\Psi(x)\le 0 \quad \text{ for all } \quad x \in \R,$$ 
by a classical sweeping argument we then get that  $\gamma^*=0$ and thus get the following contradiction 
$$0<v_\tau \le 0.$$
Therefore \eqref{cdm-eq-goal1} holds true, meaning that  $\lambda_p(c\df +\m+{\bf b})\le 0$.
Let us now prove that the inequality \eqref{cdm-eq-goal2} holds true as well.
\subsubsection*{Step Two : $\lambda_p^{\prime\prime}(-c\df +\ms+{\bf b})\le \lambda_p(c\df +\m+{\bf b})$ }

Observe that in order to prove that \eqref{cdm-eq-goal2} holds true, since by Proposition \ref{cdm-prop-lplp*}
$ \lambda_p(c\df +\m+{\bf b})=\lambda_p(-c\df +\ms+{\bf b})$
and by the definition of $\lambda_p^{\prime\prime}$, it is then enough to prove the existence of  $\varphi_p\in C^1(\R), \varphi_p$ a principal eigenfunction  associated with $\lambda_p(-c\df +\ms+{\bf b})$ such that $\varphi_p\in W^{1,\infty}\cap C^{1}(\R)$.
To do so, thanks to the  a priori regularity provided by the equation, it is enough to show that such $\varphi_p$ exists and that  $\varphi_p \in L^{\infty}$.   

%
%
%
%
%


Consider now the increasing sequence  $(R_n)_{n\in\N}:=(R_0+n)_{n\in\N}$  and the sequence $\lambda_p(-c\df +\mbs{R_n}+{\bf b})$ associated to the operator  $\ds{-c \df +\mbs{R_n}+{\bf b}}$ defined in $(-R_n,R_n)$ for $\varphi\in C([-R_n,R_n])\cap C^1((-R_n,R_n))$. 
Let $(\varphi_n)_{n\in \N}$ be the sequence of  positive principal eigenfunction associated with $\lambda_p(-c\df +\mbs{R_n}+{\bf b})$.
Such sequence $(\varphi_n)_{n\in \N}$ exists thanks to \cite{Coville2017a,Coville2020} and 
moreover $\lambda_p(-c\df +\mbs{R_n}+{\bf b})\to \lambda_p(-c\df+\ms+{\bf b})$ as $n\to \infty$.

In addition, for all $n\ge 0, \varphi_n$ satisfies
\begin{equation}
\label{cdm-eq-lim-rn1}
-c\opdf{\varphi_n}(x)+\opmbs{\varphi_n}{R_n}(x)+(b(x)+\lambda_p(c\df +\mb{R_n}+{\bf b})\varphi_n=0 \quad \text{in}\quad (-R_n,R_n).
\end{equation}

Let us recall that $b(x)<-\nu$ for all $|x|>R_0$. Therefore since by Proposition \ref{cdm-prop-lplp*} and the previous step $\lambda_p(-c\df+\ms+{\bf b})=\lambda_p(c\df+\m+{\bf b})\le 0$
and since $\lambda_p(-c\df +\mbs{R_n}+{\bf b})\to \lambda_p(-c\df+\ms+{\bf b})$   there exists  $R_1>0$ such that  
$$b(x)+\lambda_p(-c\df +\mbs{R_n}+{\bf b})\le -\frac{\nu}{4} \quad \text{for}\quad |x|\ge R_1.$$

Let us now consider $\psi(x):=2e^{-\alpha(|x|-R_1)}$ where  $\alpha$ will be chosen later on.
By a straightforward computation, we  see that for all  $R>R_1$ and $|x|\ge R_1$
\begin{align*}
 -c\opdf{\psi}(x)+\opmbs{\psi}{R_n}(x)+(b(x)+\lambda_p(-c\df +\mbs{R_n}+{\bf b}))\psi(x)&\le h(\alpha)\psi(x),
\end{align*} 
with $$ h(\alpha):=\left(|c|\alpha+\int_{\R}J(z)e^{\alpha|z|}dz-1-\frac{\nu}{4}\right).$$
Since $J$ is compactly supported, by the Lebesgue Theorem, the function $h$ is  continuous  and we can see that $h(0)=-\frac{\nu}{4}$.
By assumption $\nu>0$, and by continuity of $h$ there exists $\alpha_0>0$ such that $h(\alpha_0)<0$.
Thus,  for $\alpha=\alpha_0$ we achieve 
\begin{equation}\label{cdm-eq-rn-supersol}
-c\opdf{\psi}(x)+\opmbs{\psi}{R_n}(x)+(b(x)+\lambda_p(-c\df +\mbs{R_n}+{\bf b}))\psi(x)\le 0\quad \text{ for }\quad |x|\ge R_1.
\end{equation}
Recall that by construction, the function  $\varphi_n$ satisfies :
\begin{equation}\label{cdm-eq-rn-subsol}
-c\opdf{\varphi_n}(x)+\opmbs{\varphi_n}{R_n}(x)+(b(x)+\lambda_p(-c\df +\mbs{R_n}+ {\bf b}))\varphi_n(x)= 0\quad \text{ for $x$ in }\quad (-R_n,R_n).
\end{equation}

Since $R_n\to \infty$, there exists $n_0$ such that for all $n\ge n_0$ $R_n>R_1$. Up to a rescaling, without loss of generality, we can assume that for all $n\ge n_0$, $\ds{\sup_{[-R_1,R_1]}\varphi_n=1}$.
Therefore for all $n\ge n_0$  we get
\begin{align*}
&-c\opdf{\varphi_n}(x)+\opmbs{\varphi_n}{R_n}(x)+(b(x)+\lambda_p(-c\df +\mbs{R_n}+ {\bf b}))\varphi_n(x)= 0&\quad \text{ for all } \quad &x \in  (-R_n,R_n)\\
&-c\opdf{\psi}(x)+\opmbs{\psi}{R_n}(x)+(b(x)+\lambda_p(-c\df +\mbs{R_n}+{\bf b}))\psi(x)\le 0&\quad \text{ for all } \quad &|x|\ge R_1\\
&\varphi_n<\psi(x)  &\quad \text{ for all }\quad &x\in  [-R_1,R_1]
\end{align*}  

We claim that 
\begin{claim}
For all $n\ge n_0$, then $\varphi_n\le \psi$.
\end{claim}
 
Assume for the moment the claim holds true then we can readily finish our proof by arguing as follows. 
Since $\varphi_n\le \psi <2e^{\alpha_0 R_1}$, by using the local regularity we can see that the sequence  $(\varphi_n)_{n\in \N}$ is bounded uniformly in $C^{1,1}(\R)$. Therefore by  using a diagonal extraction and that $\ds{\sup_{[-R_1,R_1]}\varphi_n=1}$, from the sequence $(\varphi_n)_{n\in \N}$  we can extract a subsequence that converges to a nontrivial smooth function $\varphi\ge_{\not\equiv} 0$ in the  $C^{1}_{loc}(\R)$ topology. Moreover, $\varphi\le \psi$ satisfies
$$-c\opdf{\varphi}+\opms{\varphi}(x)+ (b(x)+\lambda_p(-c\df+\ms+ {\bf b}))\varphi(x)= 0 \quad \text{ for all } \quad x \in  \R,$$
which since $\lambda_p(-c\df +\ms+{\bf b})\le 0$ enforces 
$$ -c\opdf{\varphi}(x)+\opms{\varphi}(x)+b(x)\varphi(x)\ge 0\quad \text{ for all } \quad x \in  \R. $$  
By using the strong maximum principle, we see that $\varphi>0$ and from the local regularity since $\varphi\in L^{\infty}$ we see that $\varphi\in W^{1,\infty}$. Hence  $(\varphi,\lambda_p(c\df+\m+ {\bf b}))$ belongs to the set of test function that define $\lambda_p^{\prime\prime}(c\df+\m+ {\bf b})$ and therefore 
$$\lambda_p^{\prime\prime}(-c\df+\ms+ {\bf b})\le\lambda_p(-c\df+\ms+ {\bf b})=\lambda_p(c\df+\m+ {\bf b}), $$
proving that \eqref{cdm-eq-goal2} holds true.  We get similarly $\varphi \in W^{1,1}(\R)$ by observing that $\varphi \in L^1(\R)$ since $0\le \varphi\le \psi$ and $\psi \in L^1$.
\end{proof}

Let us now prove the claim. 
\begin{proof}[Proof of the claim]
Recall that  $\varphi_n \in C^1((-R_n,R_n))\cap L^{\infty}((-R_n,R_n))$ and $\ds{\inf_{(-R_n,R_n)}\psi>0}$. Therefore, we can find $\gamma_0>0$ such that $\varphi_n\le \gamma_0 \psi$ and therefore the following quantity is then well defined: 
\begin{equation}
\gamma^*:=\inf \{\gamma\ge 0 | \varphi_n\le \gamma \psi\}.
\end{equation}
We will show that $\gamma^*\le 1$.
Assume by contradiction that  $\gamma^*>1$.  By definition of $\gamma^*$ and since $\varphi_n$ and $\psi$ are continuous in $[-R_n,R_n]$ we have $\varphi_n\le \gamma^*\psi$ in $[-R_n,R_n]$ and there exists $x_0\in [-R_n,R_n]$ such that $\varphi_n(x_0)= \gamma^*\psi(x_0)$. Since $\gamma^*>1$ and $\varphi_n<\psi$ in $[-R_1,R_1]$, we deduce that $x_0\in [-R_n,R_n]\setminus[-R_1,R_1]$. Since $\varphi_n\le \gamma^*\psi$, by the strong maximum principle we then infer that $\varphi_n<\gamma^*\psi$ in $(-R_n,R_n)$ and thus  $x_0=\pm R_n$.  Since $-c<0$,
thanks to Theorems \ref{cdm-thm-CW} we have  $\varphi_n(-R_n)=0<\gamma^*\psi(-R_n)$, so we must have $x_0=+R_n$.  
By using now \eqref{cdm-eq-rn-subsol} and \eqref{cdm-eq-rn-supersol}, we can check that the function  $w:=\gamma^*\psi -\varphi_n$ satisfies
$$-c\opdf{w}(x)+\opmbs{w}{R_n}(x)+(b(x)+\lambda_p(c\df +\mb{R_n}+ {\bf b})w<0 \quad \text{ for all }\quad |x|>R_1$$
Since $w(R_n)=0$, and since $w$ is continuous in $[-R_n,R_n]$ we then infer that
$$\liminf_{x\to R_n}\opdf{w}=\lim_{x\to R_n}-\frac{1}{c}\int_{-R_n}^{R_n}J(x-y)w(y)\,dy>0,$$
which since $w(R_n)=0$ implies that there exists $x_2\in (-R_n,R_n)$ such that $w(x_2)<0$ contradicting that $w\ge 0$ in $[-R_n,R_n]$. 
Hence $\gamma^*\le 1$ and thus $\varphi_n\le \psi$.

%
\end{proof}


 \subsection{The final argument}

  We show that  the equation \eqref{cdm-eq}  has a unique solution. 
To do so let us argue by contradiction and assume that there is another solution  $u_2$. From the above subsection, we know that $u_2 \in L^1$. Let us denote $v(x):=\sup(u(x),u_2(x))$, then from the definition of $v$, we deduce that $v$ is a  weak sub-solution to \eqref{cdm-eq}, that is we have 
$$ c\opdf{v}(x) +\opm{v}(x)+ f(x,v)(x)\ge 0 \quad \text{ for almost every} \quad x \in \R.$$

 As above set $b(x):=\frac{f(x,v(x))}{v(x)}$ and now let us consider the  operator $-c\df +\ms +{\bf b}$. 
 
 We claim that 
 \begin{claim}
 There exists a sequences $(\delta_n)_{n\in\N}$ and $(\psi_n)_{n\in \N}$ and a smooth bounded non trivial function $\psi\ge 0$ such that :  
\begin{itemize}
\item[i)] for all $\ds{n\ge 0,\delta_n>0 , \delta_{n+1}\le \delta_n, \quad  \text{and}\quad \lim_{n\to\infty}\delta_n=0}$
\item[ii)] for all $\ds{n\ge 0,\psi_n\ge 0 , \psi_n\in C^{k}(\R)\cap W^{1,\infty}(\R), \quad  \text{and}\quad \psi_n \to \psi \quad\text{ in } C^{k,\alpha}_{loc}(\R)}$
\item[iii)] for all $\ds{ n\ge 0}$
$$ -c\opdf{\psi_n}(x)+\opms{\psi_n}(x)+b(x)\psi_n(x) +\delta_n\psi_n \ge 0 \quad \text{ for all } \quad x\in\R.$$
\end{itemize} 
 \end{claim}
 
 Assume for the moment that the claim holds true, then we finish our proof of the uniqueness by arguing as follows.
Let us  multiply by $u$  the equation satisfied by $\psi_n$ and  integrate it  over $\R$,  the integration is licit since $u\in C^{1}(\R)\cap H^{1}(\R)$ and $\psi_n\in W^{1,\infty}(\R)$. We then get
\begin{equation*}
\mathfrak{I_n}:=-c\int_{\R}u\opdf{\psi_n} +\int_{\R}u\opm{\psi_n} +\int_{\R}b(x)u\psi_n +\delta_n\int_{\R}u \psi_n(x)\ge 0.
\end{equation*}
Therefore 
\begin{equation}
\liminf_{n\to \infty} \mathfrak{I_n}\ge 0
\label{cdm-eq-uni1}
\end{equation}

Since $\psi_n\in W^{1,\infty}$ and $u\in H^1(\R)$ by using integration by parts, Fubini's Theorems and the equation satisfied by $u$, we can check that
 
\begin{align*}
\mathfrak{I_n}&=c\int_{\R}\opdf{u} \psi_n +\int_{\R}\psi_n\opm{u} +\int_{\R}b(x)u\psi_n
+\delta_n\int_{\R}u \psi_n \\
&= \int_{\R}\left(\frac{f(x,v(x))}{v(x)}-\frac{f(x,u(x))}{u(x)}\right)u \psi_n+\delta_n\int_{\R}u \psi_n.
\end{align*}
 Since $\delta_n\to 0$, $u_\eps \in L^1, \psi \in L^{\infty}$ and $\psi_n \to \psi$ pointwise we get that 
$$\limsup_{n\to \infty} I_n =\int_{\R}\left(\frac{f(x,v(x))}{v(x)}-\frac{f(x,u(x))}{u(x)}\right)u \psi, $$
which using that  $u \le_{\not \equiv} v$ and $f(x,s)/s$ is decreasing implies that
 $$\int_{\R}\left(\frac{f(x,v(x))}{v(x)}-\frac{f(x,u(x))}{u(x)}\right)u \psi<0$$
and thus 
\begin{equation}
\limsup_{n\to \infty} \mathfrak{I_n}< 0
\label{cdm-eq-uni2}
\end{equation}

By combining the later estimate with  \eqref{cdm-eq-uni1} we get the following contradiction 
$$0\le \liminf_{n\to \infty} \mathfrak{I_n}\le \limsup_{n\to \infty} \mathfrak{I_n}< 0.$$
Hence $u \equiv v\equiv u_2$.

\begin{proof}[Proof of the Claim]
Thanks to Proposition \ref{cdm-prop:lp} we know that $\lambda_p^{\prime\prime}(-c\df +\ms+{\bf b})\le 0$. 
Let us now analyse two separate possibility either $\lambda_p^{\prime\prime}(-c\df+\ms+{\bf b})=0$ or $\lambda_p^{\prime \prime}(-c\df+\ms+{\bf b})<0$. In the later situation, the existence of the sequences is straightforward. Indeed from  the definition of $\lambda^{\prime\prime}_p$ there exists a test function $\psi\ge 0$, $\psi \in C^{1}(\R)\cap W^{1,\infty}(\R)$ such that $-c\opdf{\psi}+\opms{\psi} +b(x)\psi\ge 0$. We achieve i), ii) and iii) by taking  the sequence $(\psi_n)_{n\in \N}$ defined for all $n$ by $\psi_n=\psi$, and  any decreasing sequence $(\delta_n)_{n}$ that converges to $0$. 

Let  us now look at the situation $\lambda^{\prime\prime}_p(-c\df+\ms+{\bf b})=0$. 
In this situation,  let $(\lambda_n)_{n\in\N}$ be a monotone decreasing sequence of positive numbers such that $\lambda_n \to 0$ as $n\to \infty$.
By definition of $\lambda^{\prime\prime}_p$, for each $n$ there exists $\psi_n$ such that 
\begin{equation}\label{cdm-eq-uni3}
-c\opdf{\psi_n}(x)+\opms{\psi_n}(x)+b(x)\psi_n+\lambda_n\psi_n(x)\ge 0\quad \text{ for all } \quad x\in \R.
\end{equation}
By assumption for $n$ large enough, says $n\ge n_0$,  we have $\lambda_n<\nu/2$ and thus $b(x)+\lambda_n\le -\frac{\nu}{2}$ for $|x|\ge R_0$.
So by integrating \eqref{cdm-eq-uni3} over $(-R,R)$, we then get  
$$  -c\int_{-R}^R\opdf{\psi_n}+\int_{-R}^{R}\opms{\psi_n}\ge \frac{\nu}{2}\int_{-R}^R\psi_n(x)-C_n,$$
with $\ds{C_0:=\left(\frac{\nu}{2}+\sup_{x\in [R_0,R_0]}\|b(x)+\lambda_n\|_{\infty}\right)\int_{-R_0}^{R_0}\psi_n}$.\\
After integration, since $\psi_n \in L^{\infty}$, there exists $C_n$ such that 
$$ C_n\ge \frac{\nu}{2}\int_{-R}^R\psi_n.$$
Therefore, $\psi_n\in L^1(\R)$ and as a consequence $\psi_n (x) \to 0$ as $|x|\to +\infty$.

Now by arguing as in the proof of Proposition \ref{cdm-prop:lp}, we can also find $\alpha>0$ independent of $n$ such that for all $C>0$, the function $w:=Ce^{-\alpha|x|}$ satisfies for all $|x|\ge R_0$ 

$$-c\opdf{w}+\opms{w} +b(x)w +\lambda_n w +\frac{\nu}{4}w <0.$$

Let us now normalised $\psi_n$ such that $\ds{\sup_{x\in [-R_0,R_0]}\psi_n(x)=1}$ and take $C:=2e^{\alpha R_0}$. By our choice of parameter, we have 
$w \ge \psi_n$ on $[-R_0,R_0]$. Now since for $n\ge n_0$ $b(x)+\lambda_n\le 0$ for $|x|\ge R_0$ by repeating the  argument  used in the proof of Proposition \ref{cdm-prop:lp}  we see that for all $n\ge n_0$, we then have  $\psi_n\le w$ in $\R $. 

Therefore  the sequence $(\psi_n)_{n\in \N}$ is uniformly bounded in $L^{\infty}(\R)\cap C^{1}_{loc}(\R)$,  and as a consequence by a diagonal extraction  we can extract a converging subsequence,  that is there exists $\psi\ge_{\not\equiv} 0$ such that $\psi_n \to \psi $ in $C^{1}_{loc}(\R)$. Moreover we can check that $\psi \in L^{\infty}$ satisfies $-c\opdf{\psi}+\opms{\psi}+b(x)\psi=0$ and $\ds{\sup_{[-R_0,R_0]}\psi=1}$. Hence $(\psi_n)_{n\in \N},(\lambda_n)_{n\in \N}$ are our desired sequence.
\end{proof}

\section{Non-existence of a solution}\label{cdm-section-nonex}
In this section, we deal with the non-existence of positive solution to \eqref{cdm-eq}  when $\lambda_p(c\df+\m+ {\bf f_{s}(x,0)})\ge 0$, proving the necessary condition stated in Theorem \ref{cdm-thm1}  validating that the sign of this quantity is the right criteria in order to predict the survival of the population. To simplify the presentation of the proofs, we treat the two cases:  $\lambda_p(c\df+\m+ {\bf f_{s}(x,0)})> 0$ and $\lambda_p(c\df+\m+ {\bf f_{s}(x,0)})= 0$ separately, the proof in the second case being more involved.

 \subsubsection*{Case $\lambda_p(c\df+\m+ {\bf f_{s}(x,0)})> 0$:}
 In this situation we argue as follows. Assume by contradiction that a positive bounded solution $u$ exists.
 By assumption, $u$ satisfies  
\begin{equation}\label{cdm-eq-linear-soussol}
 c\opdf{u}+\opm{u}+\frac{f(x,u)}{u}u= 0 \quad \text{in}\quad \R.
\end{equation}

From Proposition \ref{cdm-prop:lp} we deduce that  $\lambda_p\left(c\df+\m+{\bf \frac{f(x,u)}{u}}\right) \le 0$. 
Therefore since $f(x,u)/u \le f_{s}(x,0) $, by using the monotonic property of $\lambda_p$ with respect to the potential, we see that
$$\lambda_p(c\df+\m+ {\bf f_{s}(x,0)})\le \lambda_p\left(c\df+\m+{\bf\frac{f(x,u)}{u}}\right)\le 0.$$
Thus, we then obtain an obvious contradiction since $\lambda_p(c\df+\m+ {\bf f_{s}(x,0)})>0$.

\subsubsection*{Case $\lambda_p(c\df+\m+ {\bf f_{s}(x,0)})=0$:}
To treat this case, we will adapt to our situation  an argument introduced in \cite{Berestycki2016a} in the case $c=0$. We argue again by contradiction. 
Assume  that  a non-negative, non identically zero, bounded solution $u$ exists. By a straightforward application of the maximum principle, since $u\not\equiv 0$ we have  $u>0$ in $\R$.
Set $a(x):=f_s(x,0)$ and $b(x):=\frac{f(x,u(x))}{u(x)}$, then  by Propositions  \ref{cdm-prop:lp} and \ref{cdm-prop-pev} we have 
$$0=\lambda_p(c\df+\m+ {\bf a})\le \lambda_p(c\df+\m+ {\bf b})\le 0.$$
Therefore we have 
\begin{equation}\label{cdm-eq-ne1}
\lambda_p(c\df+\m+ {\bf a})=\lambda_p(c\df+\m+ {\bf b})=0.
\end{equation}

Fix $R_0>0$ and let us denote $\zeta \in C(\R)$  a smooth regularisation of $\chi_{R_0/2}$ the characteristic function of the interval $(-\frac{R_0}{2},\frac{R_0}{2})$. Since 
$b<a$, we can find $\eps_0>0$ small enough such that  $b\le b+\eps_0\zeta<a$.\\ 
By Proposition \ref{cdm-prop-pev} and \eqref{cdm-eq-ne1} then have 
$$\lambda_p(c\df+\m+ {\bf b +\eps_0\zeta})=0.$$ 

Since by Theorem  \ref{cdm-thm3} we can check that $\lambda_p(-c\df+\m +{\bf b +\eps_0\zeta})=\lambda_p(c\df+\m+ {\bf b +\eps_0\zeta})$, we then have 
$$ \lambda_p(-c\df +\m +{\bf b +\eps_0\zeta})=0.$$

Now thanks to Proposition \ref{cdm-prop:lp}  there exists $\varphi>0, \varphi \in C^{1}(\R)\cap W^{1,1}(\R)\cap W^{1,\infty}(\R)$ and sequences $(\psi_n)_{n\in \N}, (\delta_n)_{n\in \N}$ and a positive function $\psi \in L^{\infty}$ such that 
\begin{itemize}
\item[i)] $$-c\opdf{\varphi}(x)+\opm{\varphi}(x)+b(x)\varphi(x) \ge 0 \quad \text{ for all } \quad x\in\R.$$
\item[ii)] for all $\ds{n\ge 0,\delta_n>0 , \delta_{n+1}\le \delta_n, \quad  \text{and}\quad \lim_{n\to\infty}\delta_n=0}$
\item[iii)] for all $\ds{n\ge 0,\psi_n> 0 , \psi_n\in C^{1}(\R)\cap W^{1,\infty}(\R), \quad  \text{and}\quad \psi_n \to \psi \quad\text{ in } C^{1,\alpha}_{loc}(\R)}$
\item[iv)] for all $\ds{ n\ge 0}$
$$ -c\opdf{\psi_n}(x)+\opms{\psi_n}(x)+(b(x)+\eps_0\zeta(x))\psi_n(x) +\delta_n\psi_n \ge 0 \quad \text{ for all } \quad x\in\R.$$
\end{itemize} 

Arguing now as in the final argument subsection of the section \ref{cdm-section-uniq} we achieve the contradiction

$$0=-\eps_0\int_{\R}\varphi(x)\psi(x)\zeta(x)\,dx<0.$$ 

%
\section{Long time Behaviour}\label{cdm-section-lgtb}

In this section, we investigate the long-time behaviour of the positive solution $u(t,x)$ of 
\begin{align}\label{cdm-eq-parab1}
&\partial_t u(t,x)=\opm{u}(t,x) +f(x-ct,u(t,x)) \quad \textrm{for }t>0,\textrm{ and } x \in \R,\\ 
&u(0,x)=u_0(x)\quad \textrm{in}\quad \R.
\end{align}
For any $u_0\in C^k(\R)\cap L^{\infty}$ or in $ C^k(\R)\cap L^{1}(\R^N)$ the existence of a smooth  solution $u(t,x)\in C^1((0,+\infty), C^{\min\{1,k\}}(\R))$ respectively $u(t,x)\in C^1((0,+\infty), C^{\min\{1,k\}}(\R)\cap L^1(\R))$   is a straightforward consequence of the Cauchy-Lipschitz Theorem and of the $KPP$ structure of the nonlinearity $f$. Note also that in the moving frame of speed $c$ the smooth solution $u_c(t,x):=u(t,x+ct)$ will be also a  solution of the following problem
\begin{align}\label{cdm-eq-parab2}
&\partial_t u_c(t,x)=\opm{u_c}(t,x)+c\opdf{u_c}(t,x) +f(x,u_c(t,x)) \quad \textrm{for }t>0,\textrm{ and } x \in \R,\\ 
&u_c(0,x)=u_0(x)\quad \textrm{in}\quad \R.
\end{align}  
Observe that the function $\tilde u(t,x):=u_c(t,x-ct)$ is also a solution to \eqref{cdm-eq-parab1} so by uniqueness of the solution of the Cauchy problem, we deduce that $u(t,x)=\tilde u_c(t,x)=u_c(t,x-ct)$.\\

Before going to the proof of the asymptotic behaviour, let us recall some useful results

\begin{lemma}\label{cdm-lem-para-sub-supersol}
Assume that $u_0$ is a sub-solution to \eqref{cdm-eq-parab2}, then the solution $u_c(t,x)$ is increasing in time. Conversely, if  $u_0$ is a super-solution to \eqref{cdm-eq-parab2}  then $u_c(t,x)$ is decreasing in time. 
\end{lemma}   

The proof of this Lemma follows from a straightforward application of the parabolic maximum principle and is left to reader. Note that it may happen that $u_0$ is sub-( super-) solution of \eqref{cdm-eq-parab1} but not of \eqref{cdm-eq-parab2} and vice versa. 

Let us now prove the asymptotic behaviour of the solution of \eqref{cdm-eq-parab} and finish the proof of Theorem \ref{cdm-thm1}. We split our analysis into two main steps, first we establish a local uniform convergence towards the steady states of the system then by adapting an argument used in \cite{Berestycki2008,Berestycki2016b} to our situation we prove  the uniform convergence. 

\subsection{Step One : Local uniform convergence}
Let us first prove that for any bounded and smooth $u_0$ then the solution $u_c(t,x)$ of \eqref{cdm-eq-parab2} converges locally uniformly to $\bar u_c(x)$ a stationary solution  of \eqref{cdm-eq-parab2}. \\
Depending on the sign of $\lambda_p(c\df + \m +{\bf a})$  this stationary solution will be either $0$ or the unique non trivial solution of \eqref{cdm-eq-parab2}.
So, let $z(t,x)$ be the solution of
\begin{align}
&\partial_t z(t,x)=\opm{z}(t,x) +f(x-ct,z(t,x)) \quad \textrm{for }t>0,\textrm{ and } x \in \R,\\ 
&z(0,x)=C\|u_0\|_{\infty}\quad \textrm{on}\quad \R.
\end{align}
Since $S(x) \in L^{\infty}$ by choosing $C$ large enough, the constant $C\|u_0\|_{\infty}$ is a super-solution of \eqref{cdm-eq-parab1}. Note that since $z(0,x)\equiv Cste$, then it is also a super-solution of \eqref{cdm-eq-parab2}.
 Therefore both $z$ and $z_{c}(t,x)$ are  decreasing function of $t$ and by the parabolic maximum principle $u(t,x)\le z(t,x)$ for all $(t,x) \in [0,+\infty)\times \R$ and $u_{c}(t,x)\le z_{c}(t,x)$ for all $(t,x) \in [0,+\infty)\times \R$. Therefore, 

\begin{equation} \label{cdm-eq-parab-1}
 \limsup_{t\to \infty}u(t,x)\le \limsup_{t\to \infty} z(t,x) \quad\text{for all } \quad x \in \R.
 \end{equation}
 
 \begin{equation}\label{cdm-eq-parab-1-c}
 \limsup_{t\to \infty}u_{c}(t,x)\le \limsup_{t\to \infty} z_{c}(t,x) \quad\text{for all } \quad x \in \R.
 \end{equation}
 
Since $z_{c}(t,x)$  is a decreasing function of $t$, and $z_{c}\ge 0$,  we get $\lim_{t\to \infty}z_{c}(t,x)= \bar z(x),$ for all $x\in \R$.  Moreover by using standard regularity estimates $z_{c}(t,x)$ converges to  $\bar z$ in $C^{1,\alpha}_{loc}(\R)$ topology and thus $\bar z$ is a bounded stationary solution of \eqref{cdm-eq-parab2}. By uniqueness of the positive stationary solution, we conclude that $\bar z=\bar u_c$.     
 
Therefore we have  
 \begin{equation}\label{cdm-eq-parab-2}
 \limsup_{t\to \infty}u_{c}(t,x)\le \bar u_c(x) \quad\text{for all } \quad x \in \R.
 \end{equation}

When $\lambda_p(c\df + \m +{\bf a})\ge 0$, then no non trivial solution exists and therefore we have $\bar u_c\equiv 0$ and $u(t,x)$ as well as $u_c(t,x)$ converges locally uniformly to $0$. 
Let us now prove that when  $\lambda_p(c\df + \m +{\bf a})< 0$, then $\bar u_c>0$ and $u(t,x) \to \bar u_c$ locally uniformly as $t\to +\infty$.

In this situation, from Subsection \ref{cdm-ss-constr1}, thanks to Remark \ref{cdm-rem-kappapsi}, there exists $R_0>0$, $\psi \in C^{2}(-R_0,R_0)\cap C([-R_0,R_0])$ such that   for all $x\in (-R_0,R_0)$ we have 
$$c\opdf{\psi}+\opmb{\psi}{R_0}+(a(x)+\lambda_p(c\df +\mb{R_0}+{\bf a})+\delta)\psi\ge \frac{d^*}{4}>0$$
 with $\ds{\lambda_p(c\df +\mb{R_0}+{\bf a}) +\delta)< \frac{\lambda_p(c\df +\mb{R_0}+{\bf a})}{2}< \frac{\lambda_p(c\df +\mb{R_0}+{\bf a})}{4}}.$ \\
Set $\gamma:= \lambda_p(c\df +\mb{R_0}+{\bf a}) +\delta)$ and let us extend $\psi$ continuously by zero outside $(-R_0,R_0)$, and let us denote $\bar \psi$ this extension. One one hand since by definition  $\bar\psi\in C^{0,1}(\R)\cap C^{2}((-\infty,R_0))$ we have $\bar \psi\equiv 0$ in $\R\setminus(-R_0+\tau,R_0]$ it follows that  
\begin{equation}
c\opdf{\bar\psi}+\opm{\bar\psi}+(a(x)+\gamma)\bar \psi = \int_{\R}J(x-y)\bar \psi(y)\,dy\ge 0 \quad\text{ for all }\quad x\in \R\setminus(-R_0+\tau,R_0].
\end{equation}
On the other hand,  since $\bar \psi \ge 0$ and $\opm{\bar \psi}\ge \opmb{\bar \psi}{R_0}$, we then have 
\begin{equation}
c\opdf{\bar\psi}+\opm{\bar\psi}+(a(x)+\gamma)\bar \psi\ge \frac{d^*}{4}>0\quad\text{ for all }\quad x\in (-R_0, R_0)
\end{equation}

Therefore, $\bar \psi$ satisfies
 \begin{equation}
 c\opdf{\bar\psi}+\opm{\bar\psi}+(a(x)+\gamma)\bar \psi\ge 0 \quad\text{ for almost every }\quad x\in \R
\end{equation}

Pick now $\zeta_{\iota}\in C^{\infty}_c(\R)$ a positive symmetric mollifier whose support is include in $[-\iota,\iota]$ and define $\psi_{\iota}:=\zeta_\iota \star \bar \psi$.

Then from the above equation  we can check that $\psi_\iota$ satisfies for all $x\in\R$
$$ c\opdf{\psi_{\iota}}+\opm{\psi_{\iota}}+(a(x)+\gamma)\psi_{\iota}+ \int_{\R}\zeta_\tau(x-y)\left[a(y) - a(x)\right]\bar \psi(y)\,dy\ge 0.$$ 
By using that $a\in C^{0,\alpha}(\R)$, there exists $L_0$ such that for all $x,y\in\R$,  we have $|a(x)-a(y)|\le L_0|x-y|^{\alpha}$ and therefore from the above inequality we deduce that

$$c\opdf{\psi_{\iota}}+\opm{\psi_{\iota}}+\left(a(x)+\frac{\gamma}{2}\right)\psi_{\iota}\ge\left( -\frac{\gamma}{2} - 2L_0\iota^{\alpha}\right)\psi_{\iota}.$$ 

Then by taking $\iota$ small enough, says $\iota\le\iota_0:=\left(\frac{-\gamma}{4L_0}\right)^{\frac{1}{\alpha}}$ we then achieve for all $x\in \R$.
$$c\opdf{\psi_{\iota}}+\opm{\psi_{\iota}}+(a(x)+\frac{\gamma}{2})\psi_{\iota}\ge 0.$$

Let us now check that for $\kappa$ small, then $\kappa \psi_{\iota}$ is a subsolution to \eqref{cdm-eq-parab2}. Indeed, thank to the regularity of $f$ and since $\psi_{\iota}$ is bounded and $f(x,0)=0$, we can find $\kappa^*$ such that for all $\kappa \le \kappa^*$ we have for all $x\in\R$ 
 $$\left|\frac{f(x,\kappa \psi_{\iota})}{\kappa \psi_{\iota}}-f_s(x,0)\right|\le -\frac{\gamma}{4}.$$
Now observe that since $f(x,0)=0$, when $\psi_{\iota}=0$,   we trivially 
have 
$$c\df\kappa\psi_{\iota}+\opm{\kappa \psi_{\iota}} +f(x,\kappa \psi_{\iota})=\opm{\psi_{\iota}}\ge 0$$
whereas for $x$ such that $\psi_{\iota}>0$ by definition we have  
$$  c\df\kappa\psi_{\iota}+\opm{\kappa \psi_{\iota}} +f(x,\kappa \psi_{\iota})= \left(\left[\frac{f(x,\kappa \psi_{\iota})}{\kappa \psi_{\iota}}-f_s(x,0)\right] -\frac{\gamma}{2}\right)\psi_{\iota}\ge -\frac{\gamma\kappa}{4}\psi_{\iota}>0.$$

Therefore
$$ c\df\kappa\psi_{\iota}+\opm{\kappa \psi_{\iota}} +f(\xi,\kappa \psi_{\iota})\ge -\frac{\gamma}{4}\psi_{\iota}\ge 0 \quad \text{ in }\quad \R.$$

For $\kappa\le \kappa^*$ let $h_{\kappa}(t,x)$ and $h_{c,\kappa}(t,c)$ be the respectively the solution of
\begin{align}
&\partial_t h_{\kappa}(t,x)=\opm{h_\kappa}(t,x) +f(x-ct,h_{\kappa}(t,x)) \quad \textrm{for }t>0,\textrm{ and } x \in \R,\\ 
&h_{\kappa}(0,x)=\kappa\psi_{\iota}\quad \textrm{on}\quad \R.
\end{align}
and 
\begin{align}
&\partial_t h_{c,\kappa}(t,x)=c\opdf{h_{c,\kappa}}(t,x)+\opm{h_{c,\kappa}}(t,x) +f(x,h_{c,\kappa}(t,x)) \quad \textrm{for }t>0,\textrm{ and } x \in \R,\\ 
&h_{c,\kappa}(0,x)=\kappa\psi_{\iota}(x)\quad \textrm{on}\quad \R.
\end{align}

By definition since $\psi_{\iota}$ is bounded we can find $\kappa_0$ such that for all $\kappa\le \kappa_0$, $\kappa\psi_{\iota}\le C\|u_0\|_{\infty}=z_{c}(0,x)=z(0,x)$ and by a straightforward application of the parabolic  comparison principle, we see that for all $t>0$ and $x\in \R$
 
$$h_\kappa(t,x)\le z(t,x) \qquad \text{ and } \qquad h_{c,\kappa}(t,x)\le z_c(t,x).$$

Thanks to Lemma \ref{cdm-lem-para-sub-supersol} the function $h_{c,\kappa}$ is monotone increasing and therefore $z_c(t,x)>\kappa \psi_{\iota}$ for all times and $x$. As a consequence the stationary solution $\bar z=\bar u_c>0$ is the unique non trivial stationary solution of  \eqref{cdm-eq-parab2}. Similarly, since $h_{c,\kappa}(t,x)$ is increasing and uniformly bounded by $C\|u_0\|_{\infty}$, the positive function $\lim_{t\to \infty}h_{c,\kappa}(t,x)= \bar h(x),$ is well defined for all $x\in \R$  and  by standard regularity estimates  we can check that $\bar h$  the unique positive stationary solution of \eqref{cdm-eq-parab2} that is $\bar h=\bar u_c$. In addition, we also have $h_{c,\kappa}(t,x)\to \bar u_c(x)$ in $C^{1,\alpha}_{loc}(\R)$ as $t\to +\infty$.

Lastly, let us remark that thanks to the strong maximum principle, we have $u(1,x)>0$ so since $\psi_{\iota}$ is compactly supported and bounded we can find $\kappa_2$ such that $u(1,x)\ge \kappa_2 \psi_\iota(x)$ for all $x$. Therefore, by using the uniqueness of the solution of the Cauchy problem and the comparison principle it then standard to obtain that for all $t>0$ and $x \in \R$ $u_c(t+1,x)\ge h_{c,\kappa_2}(t,x)$. Hence we have 

\begin{equation}\label{cdm-eq-parab-3}
u_c(x)=\liminf_{t\to +\infty} h_{c,\kappa_2}(t,x)\le \liminf u(t,x) \quad \text{ for all}\quad x\in \R. 
\end{equation}

By collecting \eqref{cdm-eq-parab-2} and \eqref{cdm-eq-parab-3} we get for all $x\in\R$ 

\begin{equation}\label{cdm-eq-approx-parab-5}
 \bar u_{c}(x) \le \liminf_{t\to \infty}u_{c}(t,x)   \le \limsup_{t\to \infty}u_{c}(t,x) \le \limsup_{t\to \infty}z_{c}(t,x)=\bar u_{c}(x).
 \end{equation}

\subsection{Step Two: Uniform convergence}
Now, to complete the proof it remains to show that $\|u_{c}-\bar u_{c}\|_{\infty}\to 0$ as $t\to \infty$. To this end, we follow an argument used in \cite{Berestycki2008,Berestycki2016b}. We argue by contradiction and assume that there exists $\varepsilon>0$ and  sequences $(t_n)_{n\in\N}\in\R^+$, $(x_n)_{n\in \N}\in\R$
such that
\begin{equation}
\lim_{n\to\infty}t_n=\infty,\quad\quad|u_{c}(t_n,x_n)-\bar u_{c}(x_n)|>\varepsilon,\quad\quad\forall n\in\N.\label{cdm-eq-approx-parab-6}
\end{equation}
By \eqref{cdm-eq-approx-parab-5}, we already know that $u_{c}\to \bar u_{c}$ locally uniformly in $\R$, so, without loss of generality, we can assume that $|\xi_n|\to\infty$. 
From the construction of $\bar u_{c}$,  subsection \ref{cdm-ss-constr2}, we have $\lim_{|x|\to \infty}\bar u_c(x)= 0$. Therefore, for some $R_0>0$, we have $\bar u_{x}(x)\le \frac{\eps}{2}$ for all $x\ge R_0$. This, combined with \eqref{cdm-eq-approx-parab-5} and \eqref{cdm-eq-approx-parab-6}  enforces
\begin{equation}
z_{c}(t_n,x_n)-\bar u_{c}(x_n)\ge u_{c}(t_n,x_n)-\bar u_{c}(x_n)>\eps,\quad\quad\forall n\in\N.\label{cdm-eq-approx-parab-7}
\end{equation}
\\
Next we require the following limiting result
\begin{lemma}
For all sequences $(t_n)_{n\in \N},(x_n)_{n\in \N}$ such that $\lim_{n\to \infty }t_n =\lim_{n\to\infty}|x_n|= +\infty$, we have 
$z_{c}(t_n,x_n)\to 0$.
 \end{lemma}
\noindent Assume for the moment that the Lemma holds. Then we obtain a straightforward  contradiction since :
$$
0=\lim_{n\to \infty} z_{c}(t_n,x_n)-\bar u_{c}(x_n) \ge \lim_{n\to\infty }u_{c}(t_n,x_n)-\bar u_{c}(x_n)>\epsilon.$$
\\ 
We now prove the Lemma. 
\begin{proof}
Again, we argue by contradiction and assume that there exists $\eps>0$ and sequences $(t_n)_{n\in\N},(x_n)_{n\in \N}$ satisfying $\lim_{n\to \infty}t_n=\lim_{n\to \infty}|x_n|=\infty$ such that $z(t_n,x_n)>\eps$ for all $n \in \N$.
  Let us define $z_n(t,x):=z_c(t,x+x_n)$. It satisfies
\begin{align*}
&\partial_t z_n(t,x)=c\opdf{z_n}(t,x)+\opm{z_n}(t,x) +f(x+x_n,z_n(t,x)) \quad \textrm{for }t>0,\textrm{ and } x \in \R,\\ 
&z_n(0,x)=C\|u_0\|_{\infty}\quad \textrm{on}\quad \R^N,
\end{align*}
  and $ 0<z_n(t,x)<C\|u_0\|_{\infty}$ for $t>0$ and $x\in \R^N$. Since for all $n$, $z_n(0,x)\in C^{\infty},$ by the Cauchy Lipschitz Theorem, we see that  $z_n\in C^{1}(\R^+,C^{1,1}(\R))$. Thus, there exists $C_0>0$ independent of $n$ such that  $\|z_n\|_{C^{1,1}(\R^+, C^{1,1}(\R))} <C_0$.  
From these estimates, the sequence  $(z_n)_{n\in \N}$ is  uniformly bounded in $C^{1,1}((0,T),C^{1,\alpha}(\R^N))$ for any $T>0$.
By a diagonal extraction,  there exists a subsequence of $(z_n)_{n\in\N}$ that converges  locally uniformly to $\tilde z$. Moreover, thanks to $\lim_{|x|\to \infty}\frac{f(x,s)}{s}<0$, there exists $\kappa>0$ such that $\tilde z$ satisfies
 \begin{align}\label{cdm-eq-approx-parab-8}
&\partial_t \tilde z(t,x)\le c\opdf{\tilde z}(t,x)+\opm{\tilde z}(t,x) -\kappa \tilde z(t,x) \quad \textrm{for }t>0,\textrm{ and } x \in \R,\\ 
&\tilde z(0,x)=C\|u_0\|_{\infty}\quad \textrm{on}\quad \R.
\end{align}
\\
In addition, for all $t>0$, $\tilde z(t,0)=\lim_{n\to\infty}z_n(t,0)\ge \eps$. Since $\tilde z(0,x)$ is a super-solution of \eqref{cdm-eq-approx-parab-8}, by Lemma \ref{cdm-lem-para-sub-supersol} the function $\tilde z(t,x)$ is monotone decreasing in time. By sending $t\to \infty$, since $\tilde z\ge 0$, $\tilde z$ converges locally uniformly to a non-negative  function $\bar z$ that satisfies
\begin{align*}\label{cdm-eq-approx-parab-9}
&c\opdf{\bar z}+\opm{\bar z} -\kappa \bar z\ge 0 \quad \text{ in } \quad  \R,\\
&0\le \bar z\le C\|u_0\|_{\infty},\\
&\bar z(0)\ge \eps.
\end{align*}
Let us now consider the function $w(x):=\frac{\eps}{2} \cosh(\alpha x)-\bar z$ with $\alpha$ to be chosen. A short computation shows that $w$ satisfies
$$
c\opdf{w}(x)+\opm{w}(x) -\kappa w (x)\le  \frac{\eps}{2}\cosh(\alpha x)\left(c\alpha \tanh(\alpha x) +\int_{\R}J(z)e^{\alpha z }\,dy -1-\kappa\right) \quad \text{ for }\quad x \in   \R.
$$
The left hand side of the inequality is well defined and continuous with respect to $\alpha$ since $J$ is compactly supported. 
Since  $\int_{\R}J(z)dz=1$ and $\alpha|\tanh(\alpha x)|\le \alpha$, we can find $\alpha$ small such that
$$
c\opdf{w}+\opm{w} -\kappa w < 0 \quad \text{ in } \quad  \R.
$$ 
By construction, since $\bar z$ is bounded,  $\lim_{|x|\to \infty}w(x)=+\infty$ and  $w$ achieves a minimum in $\R^N$, say at $x_0$.
Since $w(0)=\frac{\eps}{2}-\bar z(0)\le -\frac{\eps}{2}$, we have $w(x_0)<0$.  At this point, we get the following contradiction 
  $$
0<\int_{\R}J(x_0-y)[w(y)-w(x_0)]\,dy  -\kappa w (x_0)< 0.
$$

 \end{proof}
 
\section{Fat tailed dispersal kernel}\label{cdm-s-fattailed}
In this last section, we look at the impact of the tail of the kernel $J$ and prove Theorem \ref{cdm-thm4}. We split this section into two subsection, each one dedicated respectively to the construction of non trivial solution to \eqref{cdm-eq}  and  to the existence of a threshold speed $c^{**}$ for which no positive solution to \eqref{cdm-eq} can exists. In this section, we will always assume that  $\sup_{\R}\partial_sf(x,0)<1$ and $J$ is symmetric. 
\subsection{Existence of solution}
In this subsection,  we will show that for any symmetric kernel $J$ there exists $c(J)$ such that   for all $|c|\le c(J)$ then \eqref{cdm-eq} has a solution.
More precisely,

\begin{lemma}\label{cdm-lem-fat1}
Assume that $f$ satisfy \eqref{hypf1}-\eqref{hypf3}. Assume further that $\ds{\sup_{x\in \R}\partial_s f(x,0)>1}$. Then for any symmetric kernel $J$ that satisfies \eqref{hypj1} -\eqref{hypj2} then, there exists $0<c^*(J)$ such that
for all $|c|<c^*$ there exists a positive solution to \eqref{cdm-eq}.
\end{lemma}

\begin{remark}
Observe that there is no condition on the tail of $J$ for the existence of a positive solution to \eqref{cdm-eq}.
\end{remark} 
\begin{proof}
Let $\zeta \in C^{\infty}_{c}(\R)$ be a cut-off function such that $\zeta(z)=1$ for all $|z|\le 1$, $\zeta(z)= 0$ for all $|z|>2$, $\zeta(z)=\zeta(-z)$ and $\zeta'(z)\le 0$ for all $z>0$. \\ 
For $N\in\N$ define now the function $\ds{\zeta_N:=\zeta\left(\frac{z}{N}\right)}$, the kernel $J_N(z):=J(z)\zeta_N(z)$ and the operator $\mb{N}$ standing for the operator $\M$ with the kernel $J_N$. By definition, we have 
$J_N\le J$ for all $N$ and $(J_N)_{N\in\N}$ is an increasing sequence of kernel such that $J_N\to J$ pointwise. 
 We now consider  the following approximated problem:
 \begin{equation}\label{cdm-eq-appN}
 c\opdf{u_N} + \opmb{u_N}{N} +f(x,u_N)=0 \quad \text{ in } \quad \R.
 \end{equation}
Since $\ds{\sup_{x\in \R}\partial_s f(x,0)>1}$, thanks to the Proposition 3.2 of \cite{Berestycki2016a},  we know that for all $N$,
$$\lambda_p(\mb{N}+ {\bf \partial_s f(x,0)})\le -\sup_{x\in \R}(1-\partial_s f(x,0))<0.$$
As a consequence, by Theorem \ref{cdm-thm2-bis} since $J_N$ is compactly supported and symmetric, there exists $0<c_N^*$ such that 
for all $|c|\le c^{*}_N$ the problem \eqref{cdm-eq-appN} has a positive solution and 
$$\lambda_p(c\df+\mb{N}+ {\bf \partial_s f(x,0)})<0.$$
Define $c^*:=c^*_{1}$, we claim that 
\begin{claim}
For all $N\ge 1$ there exists a unique positive solution to \eqref{cdm-eq-appN} for all $|c|<c^*$.  
\end{claim}

Assume for the moment that the claim holds true and let us finish our argumentation. 
Let us fix $c$ such that $|c|<c^*$. By the above claim, for all $N$ the equation  \eqref{cdm-eq-appN} admits a unique positive solution $u_N$.
Note that since $J_N$ is an increasing sequence, $u_N$ then satisfies 
$$
c\opdf{u_N} + \opmb{u_N}{N+1} +f(x,u_N)\ge 0 \quad \text{ in } \quad \R.
$$
As a consequence, since $u_N$ is bounded, by a standard sweeping argument, we can check that $u_N\le u_{N+1}$, meaning that the  sequence $(u_N)_{N\in \N}$ is monotone increasing. From \eqref{cdm-eq-appN} and since $f$ satisfies \eqref{hypf1}, by using the maximum principle we can find a universal positive constant $C_1$ depending on $f$ such that $\|u_N\|_{\infty}<C_1$.
By using that $c\neq 0$, we see that  $(u_N)_{N \in \N}$ is bounded  uniformly in $C_{loc}^{1,\alpha}(\R)$ and therefore by a diagonal extraction process we can extract of the non decreasing sequence $(u_N)_{N\in \N}$ an non decreasing subsequence, still denoted $(u_N)_{N\in \N}$ that converges locally uniformly to a positive bounded function $u$ which is solution to \eqref{cdm-eq}. 
The above argument being independent of $c$, we then obtain a positive solution to \eqref{cdm-eq}  for all $|c|<c^*$ and the Lemma is proved.
\end{proof}

To complete the argument of the Lemma, let us now prove the claim.
\begin{proof}[Proof of the claim]
We make an inductive argument. For $N=1$ then  since  $c^*=c^*_1$ for all $|c|<c^*$ then the problem \eqref{cdm-eq-appN} has a unique solution with $\mb{1}$. Let us assume that for some $N\ge 1$ then the problem \eqref{cdm-eq-appN} with the operator $\mb{N}$ has a unique solution for all $|c|<c^*$  and let us prove that this is still true for $N+1$.
 
Let $-c^*<c<c^*$ be fixed, then by assumption there exists $u_N>0$ a solution to \eqref{cdm-eq-appN} with the operator $\mb{N}$ and thanks to Theorem \ref{cdm-thm1} and Theorem \ref{cdm-thm3} we must have  
$$\lambda_p(c\dfrac{•}{•}+\mb{N}+ {\bf \partial_s f(x,0)})<0.$$
Since $J_N$ is an increasing sequence, thanks to the monotone behaviour of the principal eigenvalue, it follows that
 $$\lambda_p(c\df+\mb{N+1}+ {\bf \partial_s f(x,0)})\le \lambda_p(c\df+\mb{N}+ {\bf \partial_s f(x,0)}) <0.$$
 We now treat the  three cases $c>,c=0$ and $c<0$ separately.
 \subparagraph{Case $c>0:$} 
 In this situation,   by Theorem \ref{cdm-thm1} we readily conclude that there exists a unique solution to the problem \eqref{cdm-eq-appN} with $\mb{N+1}$, that is a positive solution to
 \begin{equation}\label{cdm-eq-appN+1}
 c\opdf{u} + \opmb{u}{N+1} +f(x,u)=0 \quad \text{ in } \quad \R.
 \end{equation}
 \subparagraph{Case $c=0$:}
 In this situation the existence of a positive solution is already known thanks to \cite{Berestycki2016a}.
 \subparagraph{Case $c<0$:}
 In this situation, let us observe that   by using Theorem \ref{cdm-thm3}, we have $$0>\lambda_p(c\df+\mb{N+1}+ {\bf \partial_s f(x,0)})=\lambda_p(-c\df+\mbs{N+1}+ {\bf \partial_s f(-x,0)}).$$  
 Therefore by using Theorem \ref{cdm-thm1}, we can check that  there exists a unique solution to 
 \begin{equation*}
 -c\opdf{v} + \opmbs{v}{N+1} +f(-x,v)= 0 \quad \text{ in } \quad \R.
 \end{equation*}
Now by taking $u(x):=v(-x)$, a short computation shows that $u$ is then a solution to \eqref{cdm-eq-appN+1}.

In summary, in all situations,  there exists a positive solution to \eqref{cdm-eq-appN} with the operator $\mb{N+1}$.
Since the above argument is independent of $c$, we conclude that there exists a solution to \eqref{cdm-eq-appN} with the operator $\mb{N+1}$ for all $|c|<c^*$.
The claim then follows by induction. 
\end{proof}

\subsection{An upper bound of the speed}

As for compactly supported kernel, we will obtain in this subsection  an estimate on the critical speed $c^{**}$ for fat-tailed kernel that satisfies the additional assumption 
$$\int_{\R}J(z)|z|^2\,dz<+\infty.$$
Namely, let us consider the $C^1$ function 
$$w_\tau(x):=
\begin{cases}
1-\tau x \quad \text{ when } x \le 0,\\
\frac{1}{1+\tau x} \quad \text{ when } x \ge 0.\\
\end{cases}
$$
 
To obtain our bound, it then enough to verify that for $c>>1$ we can find $\tau>0$ and $\delta>0$ such that $w_{\tau}$ satisfies:
 
$$c\opdf{w_\tau}(x)+\opm{w_{\tau}}(x)+(a(x)+\delta)w_{\tau}(x)\le 0 \quad \text{ for all }\quad  x \in \R.$$
 
Indeed, by definition of the principal eigenvalue the above inequality then implies that 
$\lambda_p(c\df +\M +{\bf a(x)})>0,$ which in turn implies the non existence of a positive solution as proved in Section \ref{cdm-section-nonex}.
 
So let us compute $\opr{w_\tau}:=c\opdf{w_\tau}(x)+\opm{w_{\tau}}(x)+a(x)w_{\tau}(x)$. For $x\le 0$, we then have
\begin{align*}
\opr{w_{\tau}}(x)&=-c\tau -\tau\int_{-\infty}^{-x}J(z)z\,dz + \int_{0}^{+\infty}J(x-y)\left[\frac{1}{1+\tau y} -1+\tau x\right]\,dy +a(x)(1-\tau x)\\
&=-c\tau +I_1+I_2+ a(x)(1-\tau x).
\end{align*}

Observe that thanks to the symmetry of $J$ we can estimate $I_1$ by 
$$I_1=-\tau\int_{-\infty}^{x}J(z)z\,dz=\tau \int^{+\infty}_{-x}J(z)z\,dz$$

Let us now estimate $I_2$, by a direct computation since $x\le 0$ we have 
\begin{align*}
I_2&=  \int_{0}^{+\infty}J(x-y)\frac{1-(1-\tau x)(1+\tau y)}{1+\tau y}\,dy\\
&=\tau\int_{0}^{+\infty}J(x-y)\frac{x-y}{1+\tau y}\,dy+ \tau x\int_{0}^{+\infty}J(x-y)\frac{\tau y}{1+\tau y}\,dy\\
&=-\tau\int_{-x}^{+\infty}J(z)\frac{z}{1+\tau (x+z)}\,dy+ \tau x\int_{0}^{+\infty}J(x-y)\frac{\tau y}{1+\tau y}\,dy\\
&\le 0.
\end{align*}
Therefore we have 
\begin{equation}\label{eq-wtau-1}
\opr{w_{\tau}}(x)\le \tau\left(-c+\int_{-x}^{+\infty}J(z)z\,dz\right)+ a(x)w_{\tau}(x).
\end{equation}

On the other hand, for $x\ge 0$ we have 

\begin{align*}
\opr{w_{\tau}}(x)&=\left(-\frac{c\tau}{1+\tau x} +\int_{-\infty}^{0}J(x-y)\left[\tau(x-y)-\tau^2 xy\right]\,dy +\int_{0}^{+\infty}J(x-y)\left[\frac{1+\tau x}{1+\tau y} -1\right]\,dy +a(x)\right)w_\tau(x)\\
&=\left(-\frac{c\tau}{1+\tau x} +I_3+I_4+ a(x)\right)w_\tau(x).
\end{align*}
Let us estimate $I_3$ and $I_4$. 
First observe that by a direct computation we have 

$$I_4=\int_{0}^{+\infty}J(x-y)\frac{\tau (x-y)}{1+\tau y}=-\tau \int_{-x}^{+\infty}J(z)\frac{z}{1+\tau(x+z)}\,dz\le \tau\int_{-x}^0J(z)z\,dz=\tau \int_{0}^xJ(z)z\,dz.$$

Let us estimate $I_3$, again a direct computation gives

\begin{align*}
I_3=\int_{-\infty}^{-x}J(z)\left[-\tau z-\tau^2 x(x+z)\right]\,dz&=\tau\int^{+\infty}_{x}J(z)z\,dz-\int^{+\infty}_{x}J(z)\left[\tau^2 x(x-z)\right]\,dz\\
&\le \tau\int^{+\infty}_{x}J(z)z\,dz+\tau^2 x\int^{+\infty}_{x}J(z)z\,dz\\
&\le \tau\int^{+\infty}_{x}J(z)z\,dz+\tau^2\int^{+\infty}_{x}J(z)z^2\,dz.
\end{align*}

Therefore we have 

\begin{equation}\label{eq-wtau-2}
\opr{w_{\tau}}(x)\le \left(-\frac{c\tau}{1+\tau x} +\tau M_1 +\tau^2M_2+ a(x)\right)w_\tau(x),
\end{equation}
where the constant $M_1$ and $M_2$ refers to 
$$M_1:=\int_{0}^{+\infty}J(z)z\,dz,\qquad M_2:=\int_{0}^{+\infty}J(z)z^2\,dz. $$

Now since $f$ satisfies \eqref{hypf1}-\eqref{hypf3} then there exists $R_0>0,\delta>0,\kappa>0$ such that $a(x)+\delta \le -\kappa$ for all $|x|\ge R_0$.
Let $\tau_0:=\frac{-M_1+\sqrt{M_1^2+4\kappa M_2}}{2M_2}$, then $\tau_0$ satisfies 
$\ds{\tau^2M_2+ \tau M_1-\kappa =0}$ and for any $c\ge 0$ we can check from \eqref{eq-wtau-2} that for $x\ge R_0$  

\begin{equation}\label{eq-wtau-3}
\opr{w_{\tau_0}}(x)+\delta w_{\tau_0}(x)\le -\frac{c\tau}{1+\tau x}w_{\tau_0}(x)\le 0.
\end{equation}
Similarly thanks to \eqref{eq-wtau-1} for $\ds{c\ge c_0:= \int_{0}^{+\infty}J(z)z\,dz}$ we have for all $\tau\ge 0$ and all $x\le -R_0$
\begin{equation}\label{eq-wtau-4}
\opr{w_{\tau_0}}(x)+\delta w_{\tau_0}(x)\le -\tau \int^{-x}_{0}J(z)z\,dz+(a(x)+\delta)w_{\tau_0}(x)\le 0,
\end{equation}

To conclude, it remains to find $c$ large such that the inequality holds for $|x|\le R_0$ and $\tau =\tau_0$.
By taking $\ds{c\ge c_1:=c_0+\sup_{x\in \R}(a(x)+\delta)\frac{1+\tau_0R_0}{\tau_0}}$ and by using  \eqref{eq-wtau-1} we can check that  for $-R_0\le x\le 0$ we have 
 \begin{equation}\label{eq-wtau-5}
\opr{w_{\tau_0}}(x)+\delta w_{\tau_0}(x)\le -\tau_0 \int^{-x}_{0}J(z)z\,dz+ (a(x)+\delta)w_{\tau_0}(x) -\sup_{x\in \R}(a(x)+\delta) \sup_{x\in [0,R_0]}w_{\tau_0}(x)\le 0,
\end{equation}

Whereas for $c\ge c_2:= (\kappa+\delta+\sup_{x\in\R}a(x))\frac{1+\tau_0R_0}{\tau_0}$ and for $0\le x\le R_0$, thanks to \eqref{eq-wtau-2}, we have
\begin{equation}\label{eq-wtau-6}
\opr{w_{\tau_0}}(x)+\delta w_{\tau_0}(x)\le \left(-\frac{c\tau_0}{1+\tau_0 R_0} +\kappa+ \delta+ \sup_{x\in \R}a(x)\right)w_{\tau_0}(x) \le 0.
\end{equation}
By collecting \eqref{eq-wtau-3}, \eqref{eq-wtau-4}, \eqref{eq-wtau-5} and \eqref{eq-wtau-6}, for $c\ge c^{\#}:= \sup\{c_0,c_1,c_2\}$ we can see that  the function $w_{\tau_0}$ satisfies 
$$ \opr{w_{\tau_0}}(x)+\delta w_{\tau_0}(x)\le 0 \qquad \text{ for all }\qquad x \in  \R.$$ 


\section*{Acknowledgements}

The author warmly thanks professors Salome Mart\'inez and Juan D\'avila  for their exceptional hospitality at  the
Departamento de Ingenier\'ia Matem\'atica  of the Universidad de Chile where this research have started during my last  visit to the Centro de Modelamiento Matem\'atico (CMM).  The author would also thanks Florian Patout (INRAE BioSP) for enlighting discussion on this subject.
The author acknowledges also support from  the “ANR DEFI” project NONLOCAL: ANR-13-JS01-0009. 
\section*{}
\bibliographystyle{amsplain}
\bibliography{cdm.bib}

\end{document}